\documentclass[12pt,a4paper]{article}
\usepackage{amsfonts,amssymb,amsmath,amscd,stmaryrd,latexsym,makeidx,theorem,psfrag}
\usepackage{amsmath,pstricks,latexsym,theorem,epsfig} 
\usepackage{amsfonts,amssymb,latexsym,epsfig,amsmath}
\usepackage{amsmath,latexsym,theorem} 
\usepackage{PSFigure,psfrag}
\usepackage{graphicx}
\newtheorem{Theorem}{Theorem}[section] 
\newtheorem{Definition}{Definition}[section] 
\newtheorem{Proposition}{Proposition}[section] 
\newtheorem{Lemma}{Lemma}[section] 
\newtheorem{Corollary}{Corollary}[section] 
 
\theorembodyfont{\rmfamily} 
\newtheorem{Remark}{Remark}[section] 

\newcommand{\rec}[1]{{(\ref{#1})}}

\def \R{\mathbb{R}} 
\def \C{\mathbb{C}}
\def \Z{\mathbb{Z}} 
\def\11{1\!\!1}

\newcommand{\de}{\partial}
\newcommand{\ve}{\varepsilon}
\newcommand{\weak}{ \stackrel{*}{\rightharpoonup}}

\newcommand{\arc}{\mathrm{arc}}
\newcommand{\M}{\mathcal}

\newcommand{\lapfr}{(-\Delta)^\frac{1}{2}}

\newenvironment{proof}{\noindent {\bf Proof.}}{\hfill$\square$\medskip}

\DeclareMathOperator{\sign}{sign}

\DeclareMathOperator{\loc}{loc}

\DeclareMathOperator*{\dist}{dist}

\def\XXint#1#2#3{{\setbox0=\hbox{$#1{#2#3}{\int}$ }
\vcenter{\hbox{$#2#3$ }}\kern-.6\wd0}}

\begin{document}
\title{The nonlocal Liouville-type equation in $\R$ and conformal immersions of the disk with boundary singularities}
\author{ Francesca Da Lio\thanks{Department of Mathematics, ETH Z\"urich, R\"amistrasse 101, 8092 Z\"urich, Switzerland.}  \and Luca Martinazzi\thanks{Department of Mathematics and Computer Science, Universit\"at Basel, Spiegelgasse 1, 4051 Basel, Switzerland. Supported by the Swiss National Foundation, Project n. PP00P2-144669.}}

\maketitle

\begin{abstract}
In this paper we perform a blow-up and quantization analysis of  the fractional Liouville  equation in dimension $1$. More precisely, given a sequence $u_k :\R \to \R$ of solutions to 
\begin{equation}
 (-\Delta)^\frac12 u_k =K_ke^{u_k}\quad \text{in }\R,\label{eqr}
\end{equation}
with $K_k$ bounded in $L^\infty$ and $e^{u_k}$ bounded in $L^1$ uniformly with respect to $k$, we show that up to extracting a subsequence $u_k$ can blow-up at (at most) finitely many points $B=\{a_1,\dots, a_N\}$ and either (i) $u_k\to u_\infty$ in $W^{1,p}_{\loc}(\R\setminus B)$ and $K_ke^{u_k}\weak K_\infty e^{u_\infty}+ \sum_{j=1}^N \pi \delta_{a_j}$, or (ii) $u_k\to-\infty$ uniformly locally in $\R\setminus B$ and $K_k e^{u_k}\weak \sum_{j=1}^N \alpha_j \delta_{a_j}$ with $\alpha_j\ge \pi$ for every $j$. This result, resting on the geometric interpretation and analysis of \eqref{eqr} provided in a recent collaboration of the authors with T. Rivi\`ere and on a classical work of Blank about immersions of the disk into the plane, is a fractional counterpart of the celebrated works of Br\'ezis-Merle and Li-Shafrir on the $2$-dimensional Liouville equation, but providing sharp quantization estimates ($\alpha_j=\pi$ and $\alpha_j\ge \pi$) which are not known in dimension $2$ under the weak assumption that $(K_k)$ be bounded in $L^\infty$ and is allowed to change sign.
\end{abstract}

\section{Introduction}
The compactness properties of the Gauss equation on a surface and the Liouville equation in $\R^2$ are important tools in prescribing the Gaussian curvature on a surface, a very famous problem. Building on a previous work of the authors with T. Rivi\`ere \cite{DMR}, we study the compactness properties of the fractional Liouville equation $\lapfr u=Ke^u$ in $\R$ under fairly weak and natural geometric assumptions.

Let us first recall some classical result. If $(\Sigma, g_0)$  is a smooth, closed Riemann surface with Gauss curvature $K_{g_0}$, for any $u\in C^\infty$ the conformal metric $g_u:= e^{2u}g$ has Gaussian curvature $K$ determined by the Gauss equation:
\begin{equation}\label{ld2}
-\Delta_{g_0} u=K e^{2u}-K_{{g_0}}\quad \text{on } \Sigma,
\end{equation}
where $\Delta_{g_0}$ is the Laplace-Beltrami operator on $(\Sigma, g_0)$, (see e.g. \cite{ch} for more details).
\par
In particular when $\Sigma =\Omega\subset \R^2$ or $\Sigma=S^2$ the equation \rec{ld2} reads respectively
\begin{equation}\label{ld2r}
-\Delta u=K e^{2u}  \quad \text{in }\Omega\subset\R^2
\end{equation}
and
\begin{equation}\label{ld2s}
-\Delta_{S^2} u=K e^{2u}-1,\quad \text{on } S^2\,.
\end{equation}\par
Both equations \rec{ld2r} and \rec{ld2s} have been largely studied in the literature. For what concerns e.g. the compactness properties of \eqref{ld2r}, in a seminal work H. Br\'ezis and F. Merle \cite{BM} showed among other things the following blow-up behavior:

\begin{Theorem}[Thm 3, \cite{BM}]\label{thbm}
Given an open subset of $\Omega$ of $\R^2$, assume that $(u_k)\subset L^1_{\loc}(\Omega)$ is a sequence of weak solutions to \rec{ld2r} with $K=K_k\ge 0$ and such that for some $1<p\le  \infty$
$$\|K_k\|_{L^p}\le \bar \kappa\,,\quad  \|e^{2u_k}\|_{L^{p'}}\le \bar A\,.$$
Then up to subsequences either
\begin{enumerate}
\item $u_k$ is bounded in $L^{\infty}_{\loc}(\Omega)$, or\par
\item there is a finite (possibly empty) set $B=\{x_1,\ldots,x_N\}\subset\Omega$  (the blow-up set) such that
$u_k(x)\to -\infty$ locally uniformly in $\Omega\setminus B$, and
$$K_k e^{2u_k}\stackrel{*}{\rightharpoonup} \sum_{i=1}^N\alpha_i\delta_{x_i} \quad\text{for some numbers }\alpha_i\ge{\frac{2\pi}{p'} }\,.$$
\end{enumerate}
\end{Theorem}

Here it is important to notice the requirement that $K_k\ge 0$.  Particularly interesting is the case $p=\infty$, in which case Theorem \ref{thbm} tells us that the amount of concentration of curvature $\alpha_i$ at each blow-up point is at least $2\pi$, which is \emph{half} of the total curvature of $S^2$. On the other hand, as shown by Y-Y. Li and I. Shafrir \cite{LS}, if one assumes that $K_k\to K_\infty$ in $C^0(\Omega)$, then a stronger and subtle quantization result holds, namely $\alpha_i$ is an integer multiple of $4\pi$. This result was then extended to higher even dimension $2m$ in the context of $Q$-curvature and GJMS-operators by several authors \cite{DR,mal, mar3,ndi,rob,mar4}, always under the strong assumption that the curvatures are continuous and converge in $C^0$ (sometimes even in $C^1$), but, at least in \cite{mar3,mar4} giving up the requirement that the curvatures are non-negative. The main ingredient here is that for uniformly continuous curvatures, in case of blow-up one can dilate the metrics and reduce to the case of constant curvature, proving that the curvature at a blow-up point is necessarily positive (by results of \cite{mar2,mar3}). Finally positive and constant curvature leads to spherical metrics, thanks to various classification results (e.g. \cite{CL, lin1, mar1}), and this is in turn responsible for the constant $4\pi$ (or higher-dimensional analogs) in the quantization results, i.e. at each blow-up point the total curvature of a sphere (of the appropriate dimension) concentrates.

It is natural to ask what happens if we remove \emph{both} the positivity and uniform continuity assumptions on the curvature, only relying on an arbitrary $L^\infty$ bound. We will address this question in dimension $1$, where the analog of \eqref{ld2r} is
$$(-\Delta)^\frac12 u= Ke^u,\quad \text{in }\R,$$
whose geometric interpretation will play a crucial role in having a precise understanding of the blow-up behaviour. The definitions of $L_\frac{1}{2}(\R)$ and $\lapfr$ will be given in the appendix.

\begin{Theorem}\label{trm2} Let $(u_k)\subset L_\frac12(\R)$ be a sequence of solutions to 
\begin{equation}\label{equk}
(-\Delta)^\frac{1}{2}u_k=K_ke^{u_k}\quad \text{in } \R{}
\end{equation}
and assume that for some $\bar\kappa,\bar L>0$ and for every $k$ it holds
\begin{equation}\label{boundKL}
\|K_k\|_{L^\infty}\le \bar \kappa,\quad \|e^{u_k}\|_{L^{1}}\le \bar L.
\end{equation}
Up to a subsequence assume that $K_k\stackrel{*}{\rightharpoonup}  K_{\infty}$ in $L^{\infty}(\R)$,  $K_ke^{u_k} \rightharpoonup \mu $ as Radon measures. 
Then there exists a finite (possibly empty) set $B:=\{x_1,\ldots, x_N\}\subset \R$ such that, up to extracting a further subsequence, one of the following alternatives holds:
\begin{enumerate}
\item $u_k\to u_\infty$ in $W^{1,p}_{\loc}(\R\setminus B)$ for $p<\infty$, where 
 \begin{equation} \label{liouvlimit}
(-\Delta)^\frac12 u_\infty=\mu= K_{\infty} e^{u_\infty}+ \sum_{i=1}^{N}\pi  \delta_{x_i}\quad \text{in }\R
\end{equation}
(compare to Fig. \ref{Fcase1}).
\item $u_k\to -\infty$ locally uniformly in $\R\setminus B$ and
$$\mu =\sum_{j=1}^N \alpha_j \delta_{x_j}$$
for some $\alpha_j\ge \pi$, $1\le j\le N$ (compare to Fig. \ref{Fpinch}).
\end{enumerate}
\end{Theorem}

Let us compare the above theorem with the result of Br\'ezis and Merle. The cost to pay for allowing $K_k$ to change sign is that even in case 1, in which $u_k$ has a non-trivial weak-limit, there can be blow-up, and in this case a \emph{half-quantization} appears: the constant $\pi$ in \eqref{liouvlimit} is half of the total-curvature of $S^1$. In case 2, instead we are able to recover the analogue of case 2 in the Br\'ezis-Merle theorem. On the other hand, the proof is now much more involved, as near a blow-up point regions of negative and positive curvatures can (and in general do) accumulate, and one needs a way to take into account the various cancelations. A direct blow-up approach does not seem to work because there can be infinitely many scales at which non-trivial contributions of curvature appear. In general it would be easy to prove that
$$|K_k|e^{u_k}\weak \sum_{j=1}^N \alpha_j \delta_{x_j}$$
for some $\alpha_j\ge \pi$, but removing the absolute values we need to prove that there is ``more'' positive than negative curvature concentrating at each blow-up point. This is turn will be reduced to a theorem of differential topology about the degree of closed curves in the plane, inspired by a classical work of S. J. Blank \cite{bla}, and to the blow-up analysis provided in \cite{DMR}, which will allow us to choose a suitable blow-up scale and estimate the curvature left in the other scales.

Things simplify and the above theorem can be sharpened if we assume that $K_k\ge 0$, hence falling back into a statement of Br\'ezis-Merle type.

\begin{figure}
\begin{center}
\includegraphics[width=16cm]{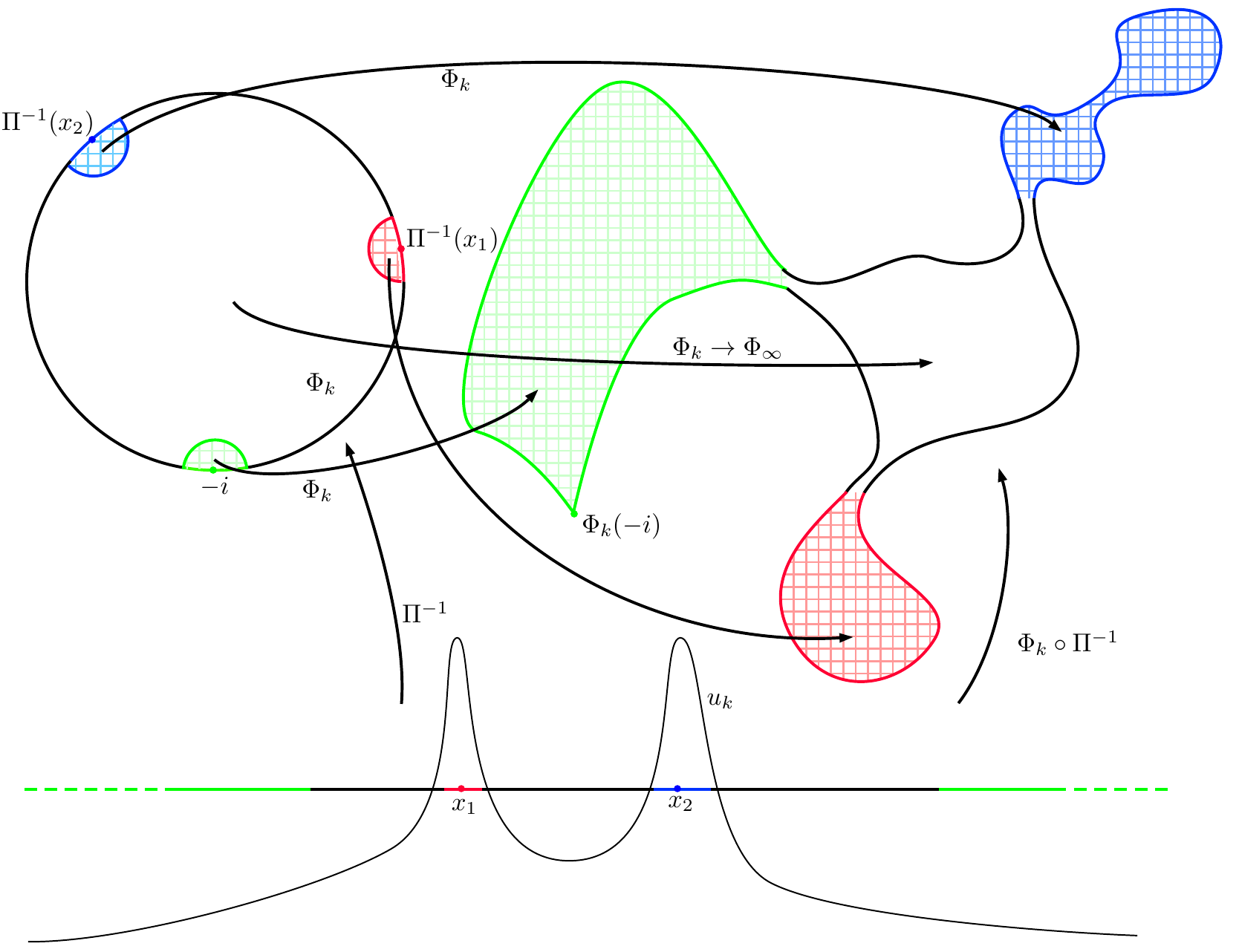}
 
\end{center}
\caption{\small The case 1 of Theorem \ref{trm2} with $N=2$, in the interpretation given by Theorem \ref{geomequiv}. $\Pi^{-1}$ is the inverse of the stereographic projection. From the function $u$ blowing up at $a_1$ and $a_2$ (and possibly at infinity, in the sense that some curvature vanishes at infinity) on construct $\Phi_k : \bar D^2\to \C$ blowing up at $a_1=\Pi^{-1}(x_1)$, $a_2=\Pi^{-1}(x_2)$ and possibly $-i$, but converging to an immersion $\Phi_\infty$ away from $\{a_1,a_2,-i\}$.}\label{Fcase1}
\end{figure}

\begin{Theorem}\label{trm3}
Let $(u_k)$ and $(K_k)$ be as in Theorem \ref{trm2} and additionally assume that $K_k\ge 0$.
Then up to a subsequence in case 1 of Theorem \ref{trm2} we have $N=0$
and in case 2 we have $\alpha_j>\pi$ for $1\le j\le N$.
\end{Theorem}

The proofs of Theorems \ref{trm2} and \ref{trm3} are strongly based on the following geometric interpretation of Equation \eqref{equk}.

\begin{Theorem}\label{geomequiv} Let  $u\in L_\frac12(\R)$ with $e^u\in L^1(\R)$ satisfy
\begin{equation}\label{liou5int}
(-\Delta)^\frac12 u=K e^u\quad \text{in }\R
\end{equation}
for some function $K \in L^{\infty}(\R)$. Then there exists $\Phi\in C^0(\bar D^2,\C)$ with $\Phi|_{S^1}\in W^{2,p}_{\loc}(S^1\setminus\{-i\},\C)$ for every $p<\infty$ such that $\Phi$ is a holomorphic immersion of $\bar D^2\setminus \{-i\}$ into $\C$,
\begin{equation}\label{cond1bis}
|\Phi'(z)|= \frac{2}{1+\Pi(z)^2} e^{u(\Pi(z))},\quad \text{for } z\in S^1\setminus\{-i\},
\end{equation}
and the curvature of the curve $\Phi|_{S^1\setminus\{-i\}}$ is $\kappa:=K\circ\Pi$, where $\Pi: S^1\setminus\{-i\}\to \R$ is the stereographic projection given by $\Pi(z)=\frac{\Re z}{1+\Im z}$.
\end{Theorem}

Theorem \ref{geomequiv} in turn is an extension of the following result proven in \cite{DMR}:

\begin{Theorem}[\cite{DMR}]\label{geomequiv0} 
A function  $\lambda\in  L^1(S^1,\C)$ with $L:=\|e^{\lambda}\|_{L^1(S^1)}<\infty$  satisfies 
\begin{equation} \label{liouvfracn1}
(-\Delta)^\frac12 \lambda=\kappa e^{\lambda}-1\quad \text{in }S^1
\end{equation}
for some function $\kappa\in L^{\infty}(S^1)$, if and only if there exists  a closed curve 
$\gamma\in W^{2,\infty}(S^1,\C)$, with   $|\dot\gamma|\equiv\frac{L}{2\pi}$, a holomorphic  immersion $\Phi\colon \bar D^2\to \C$ a diffeomorphism $\sigma\colon S^1\to S^1$    ,  such that $\Phi\circ\sigma(z)=\gamma(z)$ for all $z\in S^1\,,$
 \begin{equation}\label{cond1}
|\Phi'  (z)|=e^{\lambda(z)},~~z\in S^1\,,
\end{equation}
and the curvature of $\Phi(S^1)$ is $\kappa$\,.
\end{Theorem}

\begin{figure}
\begin{center}
\includegraphics[width=14cm]{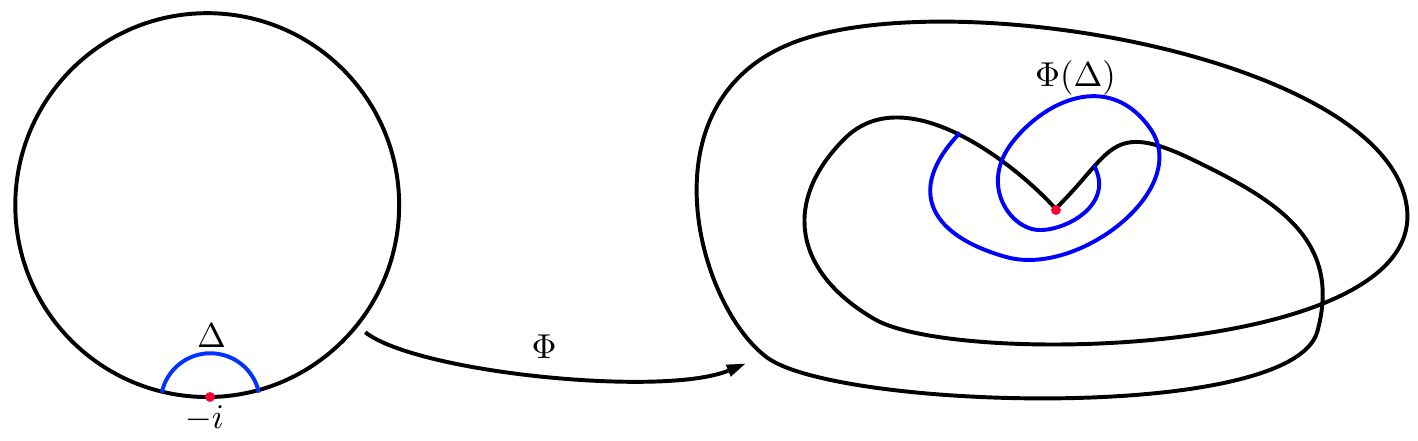}
 
\end{center}
\caption{\small The map $\Phi$ given by Theorem \ref{geomequiv} is in general singular at $-i$. Because of this the curve $\Phi|_{S^1}$ can have rotation index greater than $1$ (it is $2$ in the above example). The image of the curve $\Delta$ should facilitate the intuition of the geometry of $\Phi$ near $-i$.}\label{FPhi}
\end{figure}

Indeed it was shown in \cite{DMR} that if $u\in L_\frac12(\R)$, $e^u\in L^1(\R)$ and $u$ satisfies \eqref{liou5int}, then
$$\lambda(e^{i\theta}):=u(\Pi(e^{i\theta}))-\log(1+\sin \theta)$$
satisfies
\begin{equation}\label{liou6}
(-\Delta)^\frac{1}{2}\lambda=(K\circ\Pi) e^{\lambda}-1 +\left(2\pi -\Lambda \right)\delta_{-i}\quad \text{in }S^1, \quad \Lambda:=\int_{\R}\lapfr u dx.
\end{equation}
Then Theorem \ref{geomequiv} can be seen as an extension of Theorem \ref{geomequiv0} to the case in which $\Lambda\ne 2\pi$.


Another ingredient in the proof of Theorem \ref{trm2} comes from differential geometry and roughly speaking says the if a closed positively oriented curve $\gamma:S^1\to \C$ of class $C^1$ except at finitely many points can be extended to a function $F\in C^0(D^2, \C)$ which is a $C^1$-immersion except at finitely many boundary points, then the rotation index of $\gamma$ is at least $1$ (see Definition \ref{defrot} and Theorem \ref{trmBlank}). This is obvious if $F\in C^1(\bar D^2,\C)$ is an immersion everywhere (no corners on the boundary), and in fact the rotation index of $\gamma=F|_{S^1}$ is $1$ in this case, but in the general case the rotation index can be arbitrarily high and the proof that it must be strictly positive rests on ideas introduced by Blank to study which regular closed curves can be extended to an immersion of the disk into the plane.

Mixing this result (rotation index at least $1$) with Theorem \ref{geomequiv} (via Proposition \ref{proppull2} and Corollary \ref{corangle}) one obtains at once:

\begin{Theorem}\label{thmpi} Let  $u\in L_\frac12(\R)$ with $e^u\in L^1(\R)$ satisfy \eqref{liou5int} for some function $K \in L^{\infty}(\R)$ and let
$$\Lambda:=\int_{\R} K e^u dx.$$
Then $\Lambda \ge \pi$. If $K\ge 0$ then $\Lambda>\pi$.
\end{Theorem} 

When $K\ge 0$, the proof of Theorem \ref{thmpi} is fairly easy, since $u$ can be written as
$$u(x)=\frac{1}{\pi}\int_{\R}\log\left(\frac{1+|y|}{|x-y|}\right) K(y)e^{u(y)}dy+C_0$$
for a constant $C$ (compare to Proposition 5.5 in \cite{DMR}), and for $K\ge 0$ one easily obtain an asymptotic behavior of the form
\begin{equation}\label{uasym}
u(x)\ge -\frac{\Lambda}{\pi} \log(1+|x|) -C
\end{equation}
for some contant $C$ and this contradicts the assumption $e^{u}\in L^1(\R)$ if $\Lambda\le \pi$. More delicate is the case when $K$ is allowed to change sign, because of possible cancelations. In any case, more than result itself, we would like to emphasize that it has a geometric interpretation in terms of the rotation index of the curve $\Phi|_{S^1}$ given by Theorem \ref{geomequiv} (see Fig. \ref{Findex2}). Whether an analogue interpretation can be used for the Liouville equation in dimension $2$ or higher is unknown.

Another consequence of Theorem \ref{geomequiv} is a new and geometric proof, not relying on a Pohozaev-type identity, nor on the moving plane technique (moving point in this case), of the classification of the solutions to the non-local equation
\begin{equation}\label{liou}
(-\Delta)^\frac{1}{2} u = e^{u}\quad\text{in }\R,
\end{equation}
under the integrability condition 
\begin{equation}\label{area}
L:=\int_{\R{}}e^u dx<\infty.
\end{equation}

\begin{Theorem}[\cite{CY,Xu}]\label{trm1} Every function $u\in L_{\frac{1}{2}}(\R)$ solving \eqref{liou}-\eqref{area} is of the form
\begin{equation}\label{specsol}
u_{\mu,x_0}(x):=\log\bigg(\frac{2\mu}{1+\mu^2|x-x_0|^2}\bigg), \quad x\in \R^n,
\end{equation}
for some $\mu>0$ and $x_0\in\R{}$.
\end{Theorem}

The proof will be given at the end of Section \ref{sec:2}. Similar higher-dimensional, also in the fractional case have appeared in \cite{CY, hyd,JMMX,lin1, mar1, Xu}.

\paragraph{Acknowledgements} We would like to thank Carlos Tomei and Tristan Rivi\`ere for referring us to the works of Blank \cite{bla} and Po\'enaru \cite{poe}, and to the monograph \cite{MST}.

\section{From the Liouville equation to conformal immersions of the unit disc with a boundary singularity}\label{sec:2}

It is easy to verify that for  $\phi\in L^1(S^1)$ we have
\begin{equation}\label{lapHilb}
(-\Delta)^\frac12 \phi(\theta)=\sum_{n\in\mathbb{Z}}|n|\hat\phi(n)e^{in\theta}={\cal{H}}\left(\frac{\partial \phi}{\partial\theta}\right)=\frac{\de \mathcal{H}(\phi)}{\de\theta},
\end{equation}
where ${\cal{H}}$ is the Hilbert Transform  on $S^1$ defined by
$${\cal{H}}(f)(\theta) :=\sum_{n\in \Z\setminus\{0\}} -i \sign (n)\hat{f}(n)e^{in\theta},\quad f\in\mathcal{D}'(S^1) \,.$$

The following lemma is well-known.

\begin{Lemma}\label{LemmaHilbert}  The Hilbert transform $\cal{H}$ is  bounded from $L^p(S^1)$ into itself, for $1<p<+\infty$, and it is of weak type $(1,1)$. A function  $f:=u +i v$ with $u,v\in L^1(S^1,\R)$ can be extended to a holomorphic function in $D^2$ if and only if $v=\mathcal{H}(u)+a$ for some $a\in \mathbb{C}$ .
\end{Lemma}

Consider the conformal map $G:D^2\to \R^2$ given by
$$G(z)=\frac{iz+1}{z+i}=\frac{z+\bar z +i(|z|^2-1)}{1+|z|^2+i(\bar z-z)}.$$
 We will use  on the domain $D^2$ the coordinate $z=\xi+i\eta$ and on the target $\R^2$ the coordinates $(x,y)$ or $x+iy$.
Notice that
$$G|_{S^1}(\xi+i\eta)=\frac{\xi}{1+\eta},$$
hence $\Pi:= G^1|_{S^1}:S^1\setminus \{-i\}\to\R$ is the stereographic projection. Its inverse is
\begin{equation}\label{formulaPi}
\Pi^{-1}(x)=\frac{2x}{1+x^2}+i\left(-1+\frac{2}{1+x^2}\right).
\end{equation}
If we write $\Pi^{-1}(x)=e^{i\theta(x)}$ we get the following useful relation 
\begin{equation}\label{sintheta}
1+\sin(\theta(x))=\frac{2}{1+x^2},\quad \frac{2}{1+\Pi(\theta)^2}=1+\sin\theta,
\end{equation}
which follows easily from $\sin(\theta(x))=\Im(\Pi^{-1}(x))=\frac{1-x^2}{1+x^2}$.

The following proposition is proven in \cite[Prop. 4.2]{DMR}.

\begin{Proposition}\label{proppull2} Given $u:\R\to\R$ set $\lambda$ as
\begin{equation}\label{deflambda2}
\lambda(\theta)=u(\Pi(\theta))+\log\left|\frac{\de \Pi}{\de \theta}\right| =u(\Pi(\theta))- \log\left(1+\sin\theta\right),
\end{equation}
or equivalently and using \eqref{sintheta}
\begin{equation}\label{deflambda3}
u(x)=\lambda(\Pi^{-1}(x))+\log\left(\frac{2}{1+x^2}\right).
\end{equation}
Then $u\in L_\frac{1}{2}(\R)$ if and only if $\lambda\in L^1(S^1)$, and $(-\Delta)^\frac12 u\in L^1(\R)$ if and only if $(-\Delta)^\frac{1}{2} \lambda\in L^1(S^1\setminus\{-i\})$. In this case $u$ solves \eqref{liou5int} if and only if $\lambda$ solves
\begin{equation}\label{liou6}
(-\Delta)^\frac{1}{2}\lambda=(K\circ \Pi) e^{\lambda}-1 +\left(2\pi -\int_{\R}(-\Delta)^\frac12 u dx\right)\delta_{-i}\quad \text{in }S^1.
\end{equation}
\end{Proposition}




\noindent\textbf{Proof of Theorem \ref{geomequiv}} Given $u\in L^1(\R)$ as in the statement of the theorem, set $\lambda\in L^1(S^1)$ as in \eqref{deflambda2}, so that by Proposition \ref{proppull2} we have
$$\lapfr \lambda =\kappa e^\lambda -1+\beta\delta_{-i},$$
with
$$\beta:=2\pi-\int_{\R}\lapfr udx,\quad \kappa=K\circ\Pi.$$
By Lemma \ref{lemmafund4} the function $\lambda_1:=\lambda+\frac{\beta}{2\pi}\log(1+\sin\theta)$ satisfies 
$$(-\Delta)^\frac12 \lambda_1=\kappa e^\lambda -1+\frac{\beta}{2\pi}\quad\text{in }S^1,$$
hence by regularity theory we have $\lambda_1\in L^p(S^1)$, and $\lambda\in L^p(S^1)\cap W^{1,p}_{\loc}(S^1\setminus \{-i\})$ for every $p<\infty$. Then $\rho:=\mathcal{H}(\lambda)\in L^p(S^1)\cap W^{1,p}_{\loc}(S^1\setminus\{-i\})$ and by Lemma \ref{LemmaHilbert} one can define the holomorphic extension $\widetilde{\lambda+i\rho}$ and $\phi:=e^{\widetilde{\lambda+i\rho}}$, and 
$$\Phi(z):=\int_{\Sigma_{0,z}}\phi(w)dw,\quad z\in \bar D^2\setminus\{-i\}$$
where $\Sigma_{0,z}$ is the oriented segment joining $0$ to $z$, or any other piecewise $C^1$ path in $\bar D^2\setminus\{-i\}$ joining $0$ to $z$.
That $\Phi$ can be continuously extended to the point $-i$ depends on the following facts.
We have
$$\lim_{\delta\to 0} \bigg|\int_{-\delta}^\delta \phi (-ie^{it}) dt\bigg| \le \lim_{\delta\to 0} \int_{-\delta}^\delta e^{\lambda(-ie^{it})}  dt =0$$
since $e^\lambda\in L^1(S^1)$.
We also have
$$\lim_{r\to 1^-}\int_{S^1}|\phi(re^{it})-\phi(e^{it})|dt=0,$$
since $\phi(re^{it})$ is obtained by convolving $\phi(e^{it})$ with the Poisson kernel $P(r,t)$, see \eqref{Poisson1}. In particular
$$\lim_{\delta\to 0} \lim_{r\to 1^-}\bigg|\int_{-\delta}^\delta \phi (r(-ie^{it})) dt\bigg| =0.$$
Similarly, using that $W^{1,p}_{\loc}(S^1\setminus\{-i\})\hookrightarrow C^0_{\loc}(S^1\setminus\{-i\})$ we obtain
$$\lim_{r\to 1^-} \phi(r(-ie^{i\delta}))=\phi(-ie^{i\delta}), \quad \text{for every }\delta\in\R\setminus 2\pi\mathbb{Z},$$
so that
\begin{equation}\label{eqdelta}
\lim_{r\to 1^-} \int_r^1 |\phi(\rho(-ie^{i\delta}))|d\rho=0, \quad \text{for every }\delta\in\R\setminus 2\pi\mathbb{Z},
\end{equation}
Using these facts one can easily prove that for any sequence $z_k\to {-i}$, the sequence $(\Phi(z_k))$ is a Cauchy sequence, since one can join $z_k$ to $z_\ell$ by paths $\Sigma_{k,\ell}$ made of arcs as the one described above (pieces of radii and arcs of circles), in such a way that
$$\lim_{k,\ell\to\infty}(\Phi(z_\ell)-\Phi(z_k))=\lim_{k,\ell\to\infty}\int_{\Sigma_{k,\ell}}\phi(w)dw= 0,$$
and one can extend $\Phi$ continuously to $-i$.

For instance, consider the sequence $z_k:= -i(1-k^{-1})\to -i$. For any $k\ge \ell\ge 1$ one case choose the path $\Sigma_{k,\ell}$ made by joining the two arcs
$$\{-i(1-\ell^{-1})e^{it},\, t\in (0, \delta_\ell)\},\quad \{-i(1-k^{-1})e^{it},\, t\in (0,\delta_\ell)\}$$
and the segment
$$\{r(-ie^{i\delta_\ell}),\, r\in (1-\ell^{-1},1-k^{-1})\},$$
where $\delta_\ell\to 0$ slowly enough to have
$$\lim_{k\ge\ell\to\infty}\int^{1-k^{-1}}_{1-\ell^{-1}} |\phi(\rho(-ie^{i\delta_\ell}))|d\rho =0.$$
(One cannot simply join $z_k$ to $z_\ell$ with the segment
$$\{-ir,\,r\in (1-\ell^{-1},1-k^{-1})\}$$
because at this stage we cannot use \eqref{eqdelta} for $\delta=0$.)

\medskip


We now want to prove that $\kappa$ is the curvature $\kappa_\Phi$ of the curve $\Phi|_{S^1\setminus\{-i\}}$. We compute for $\theta\ne -\frac{\pi}{2}$
\[\begin{split}
\frac{\de\Phi}{\de\theta}(e^{i\theta})&=\Phi'(e^{i\theta})ie^{i\theta}=ie^{\lambda+i\rho+i\theta}\\
\frac{\de^2\Phi}{\de\theta^2}(e^{i\theta})&=ie^{\lambda+i\rho+i\theta}\left[\frac{\de\lambda}{\de\theta}(e^{i\theta})+i\left(\frac{\de\rho}{\de\theta}(e^{i\theta})+1\right)\right],\\
\end{split}\]
and, using a well-known formula for the curvature of a curve in the plane,
\[\begin{split}
\kappa_{\Phi}&=\left|\frac{\de\Phi}{\de\theta}\right|^{-3}\left(\frac{\de\Phi}{\de\theta}\right)^\perp \cdot \frac{\de^2\Phi}{\de\theta^2}\\
&=e^{-\lambda} \left(\frac{\de\rho}{\de\theta}+1\right)\\
&= e^{-\lambda}((-\Delta)^\frac12 \lambda+1)\\
&=\kappa,
\end{split}\]
where we used that the vectors $e^{\lambda+i\rho+i\theta}$ and $ie^{\lambda+i\rho+i\theta}$ are orthogonal, and \eqref{lapHilb}.

Finally, \eqref{cond1bis} follows at once from \eqref{sintheta} and \eqref{deflambda2}:
$$|\Phi'(e^{i\theta})|=e^{\lambda(e^{i\theta})}=\frac{e^{u(\Pi (e^{i\theta}))}}{1+\sin\theta} =\frac{2}{1+\Pi(e^{i\theta})^2}e^{u(\Pi (e^{i\theta}))},$$ 
for $e^{i\theta}\ne -i$.
\hfill$\square$

\medskip

\noindent\textbf{Proof of Theorem \ref{trm1}} Given $u$ solving \eqref{liou} according to Proposition \ref{proppull2} the function $\lambda$ defined by \eqref{deflambda2} satisfies
$$(-\Delta)^\frac12 \lambda = e^\lambda -1+\beta\delta_{-i}, $$
 with
$$\beta=2\pi-\int_{\R}(-\Delta)^\frac12 u dx\,.$$
By Theorem \ref{geomequiv} the function $\lambda$ determines a holomorphic immersion $\Phi:\bar D^2\setminus\{-i\}\to \C$ with the property that $\Phi_{S^1\setminus\{-i\}}$ is a curve of curvature $1$ which extends continuously to a closed curve $\Phi|_{S^1}$. Then up to translations $\Phi|_{S^1}$ is a parametrization of the unit circle, possibly with degree $n$ different from $1$. But with Lemma \ref{lemmabla}, together with the fact that $\Phi$ is holomorphic and $\Phi'$ never vanishes in $D^2$, we immediately get that $n=1$, $\beta=0$ and $\Phi$ is a M\"obius diffeomorphism of $\bar D^2$. From that, the explicit form of $\lambda$ and $u$ can be computed as in the proof of Theorem 1.8 in \cite{DMR}. \hfill $\square$

\section{Compactness of the singular Liouville equation in $S^1$ and proof of Theorems \ref{trm2} and \ref{trm3}}\label{sec:comps1}

In order to prove Theorems \ref{trm2} and \ref{trm3} we pull back the Liouville equation \eqref{equk} from $\R$ into $S^1$ using Theorem \ref{geomequiv}, and then prove the following theorem about a singular Liouville equation on $S^1$, whose statement can be easily pulled back onto $\R$. We remark that in dimension $2$ singular Liouville equations arise in the study of Chern-Simons vortices, see e.g. \cite{Tar}.

\begin{Theorem}\label{mainth} Let $(\lambda_k)\subset L^1(S^1)$ be a sequence with
\begin{equation}\label{boundlength}
L_k:=\|e^{\lambda_k}\|_{ L^1(S^1)}\le \bar L
\end{equation}
satisfying  
\begin{equation} \label{liouvfrack}
(-\Delta)^\frac12 \lambda_k=\kappa_k e^{\lambda_k}-1+\beta_{k}\delta_{-i} \quad \text{in }\mathcal{D}'(S^1)
\end{equation}
where $(\beta_k)\subset \R$  is a bounded sequence and $\kappa_k\in L^\infty(S^1)$ satisfies
\begin{equation}\label{boundcurv}
\|\kappa_k\|_{L^{\infty}(S^1)}\le \bar \kappa\,.
\end{equation}
Assume w.l.o.g. that $\beta_{k}\to\beta_\infty$ 
and $\kappa_k\stackrel{*}{\rightharpoonup}\kappa_{\infty}$ in $L^{\infty}(S^1)$. Then, up to a subsequence, $\kappa_k e^{ \lambda_k} \stackrel{*}{\rightharpoonup}\mu $ as Radon measures on $S^1$, and $\kappa_k e^{ \lambda_k} \to \mu$ in $W^{1,p}_{\loc}(S^1\setminus (B\cup\{-i\}))$ where $B:=\{a_1,\ldots, a_N\}$ is a (possibly empty) subset of $S^1\setminus\{-i\}$. Setting
$$\bar\lambda_k:=\frac{1}{2\pi}\int_{S^1}\lambda_kd\theta,$$
$$a_0:=-i,\quad   -\frac{\pi}{2}=\theta_0<\theta_1<\dots<\theta_N<\frac{3\pi}{2} \text{ such that }a_j=e^{i\theta_j},\quad 0\le j\le N,$$
one of the following alternatives holds:
\begin{enumerate}
\item $|\bar\lambda_k|\le C$ and $\mu=\kappa_\infty e^{\lambda_\infty}+\pi(\delta_{a_1}+\dots+\delta_{a_N})+\alpha_0\delta_{-i}$, 
$$\lambda_k\to\lambda_\infty\quad \text{ in }W^{1,p}_{\loc}(S^1\setminus (\{a_0,\dots, a_N\}),$$
with $\lambda_\infty,e^{\lambda_\infty}\in L^1(S^1)$, $\alpha_0\in \R$ and 
\begin{equation} \label{liouvlimit}
(-\Delta)^\frac12 \lambda_\infty=\kappa_{\infty} e^{\lambda_\infty}-1+(\alpha_0+\beta_\infty) \delta_{-i}+\sum_{j=1}^{N}\pi \delta_{a_j}\quad \text{in }S^1\,.
\end{equation}

\item $\bar \lambda_k\to -\infty$ as $k\to\infty$, $\mu= \alpha_0\delta_{a_0}+\dots+\alpha_N\delta_{a_N}$,
$$v_k:=\lambda_k-\bar\lambda_k \rightharpoonup v_\infty \quad \text{ in }W^{1,p}_{\loc}(S^1\setminus\{a_0,\dots, a_N\})\text{ for every }p<\infty,$$
where
$$v_{\infty}(e^{i\theta})= -\sum_{j=0}^N\frac{\alpha_j}{2\pi}\log(2(1-\cos(\theta-\theta_j))) -\frac{\beta_\infty}{2\pi}\log(2(1-\cos(\theta-\theta_0))) $$
solves
\begin{equation} \label{liouvlimit3}
(-\Delta)^\frac12 v_\infty=-1+\sum_{j=0}^N\alpha_j \delta_{a_j}+\beta_\infty\delta_{-i}\quad \text{in }S^1\,.
\end{equation}
Moreover $\alpha_j\ge\pi$ for $1\le j\le N$.\par
\end{enumerate}
If $K_k\ge 0$ for every $k$, then in case 1 we have $N=0$ and in case $2$ we have $\alpha_j>\pi$ for $1\le j\le N$. More precisely, if $K_k\ge 0$ is a neighbourhood of $a_j$ for some $j$ and for every $k$, then we are in case $2$ and $\alpha_j>\pi$.
\end{Theorem}

\begin{Remark} The terms $\alpha_0$ and $\beta_\infty$ multiplying $\delta_{-i}$ in \eqref{liouvlimit} and \eqref{liouvlimit3} have different origins: $\beta_\infty$ is the limit of the angles $\beta_k$, while $\alpha_0$ is the amount of curvature concentrating at $-i$. Both coefficients can be $0$ of course.
\end{Remark}

 
  
  

The following theorem is a preliminary step towards the proof of Theorem \ref{mainth}.

\begin{Theorem}\label{convergencephi} 
Let $(\lambda_k)$ be a sequence as in Theorem \ref{mainth}, and let $(\Phi_k)\subset C^0(\bar {D}^2, \C)$ be holomorphic immersions with $\lambda_k=(\log\left|\Phi_k'\right|)|_{S^1\setminus\{-i\}}$ and $\Phi_k(1)=0$, as given by Theorem \ref{geomequiv}). 
Then, up to extracting a subsequence, the following set is finite
\begin{equation}\label{defB}
B:=\left\{a\in S^1\setminus\{-i\}:\lim_{r\to 0^+}\limsup_{k\to\infty} \int_{B(a,r)\cap S^1} |\kappa_k| e^{\lambda_k}d\theta\ge \pi \right\}=\{a_1,\ldots,a_N\}\,,
\end{equation}
and for functions $v_\infty\in L^1(S^1,\R)$ and $\Phi_\infty\in W^{1,2}(D^2,\C)$ we have for $1\le p <\infty$
\begin{equation}\label{convlambdak}
\lambda_k-\bar\lambda_k \rightharpoonup v_\infty \quad\text{in }W^{1,p}_{\loc}(S^1\setminus (B\cup\{-i\}))\,,\quad \bar\lambda_k:=\frac{1}{2\pi}\int_{S^1}\lambda_k d\theta\,,
\end{equation}
and
$\Phi_k \to \Phi_\infty$ weakly in $W^{2,p}_{\loc}(\bar D^2\setminus (B\cup\{-i\}),\C)$, strongly in $W^{1,2}(D^2,\C)$, and in $C^0(\bar D^2,\C)$.
Moreover, one of the following alternatives holds:\par
\begin{enumerate}
\item The sequence $(\bar \lambda_k)\subset\R$ is bounded, $\lambda_k\to
\lambda_\infty$, and $\Phi_\infty$ is a holomorphic immersion of $\bar D^2\setminus (B\cup\{-i\})$. Moreover
\begin{equation}\label{limk1}
\kappa_k e^{\lambda_k}\weak \kappa_\infty e^{\lambda_\infty}+\sum_{j=0}^N\alpha_j\delta_{a_j},\quad \lapfr \lambda_\infty=\kappa_\infty e^{\lambda_\infty}-1+\sum_{j=0}^N\alpha_j\delta_{a_j} +\beta_\infty\delta_{a_0},
\end{equation}
for some $a_0,\dots, a_N\in\R$, where $a_0:=-i$.
\item $\lambda_k\to -\infty$ locally uniformly in $S^1\setminus (B\cup\{-i\})$ as $k\to +\infty,$ and $\Phi_\infty\equiv Q$ for some constant $Q\in \C$. Moreover
\begin{equation}\label{limk2}
\kappa_k e^{\lambda_k}\weak \sum_{j=0}^N\alpha_j\delta_{a_j},\quad \lapfr v_\infty= -1+\sum_{j=0}^N\alpha_j\delta_{a_j} +\beta_\infty\delta_{a_0}
\end{equation}
for some $a_0,\dots, a_N \in\R$.
\end{enumerate}
\end{Theorem}

\begin{proof}
The proof of Theorem \ref{convergencephi} is essentially identical to the proof of Theorem 3.5 is \cite{DMR}, with \cite[Lemmata 3.3, 3.4]{DMR} replaced by Lemmata \ref{eLemma} and \ref{L2infty} below, which take into account the presence of the term $\beta_k\delta_{-i}$ in the equation.

The convergence of $\Phi_k\to \Phi_\infty$ uniformly in $\bar D^2$ is not proven in \cite{DMR}, but follows easily from the convergence in $W^{2,p}_{\loc}(\bar D^2\setminus (B\cup \{-i\}))$ and Lemma \ref{L2infty} below. Indeed \eqref{stimaLo2} shows that the oscillation of $\tilde \lambda_k =\log|\Phi'_k|$ in a neighborhood of the blow-up points must be uniformly bounded, so that one can apply the Theorem of Ascoli-Arzel\'a to the sequence $(\Phi_k)$.

The proof of \eqref{limk1} and \eqref{limk2} is standard, but we sketch it. By the bounds $\|\kappa_ke^{\lambda_k}\|_{L^1}\le C$ we have (up to a subsequence) $\kappa_k e^{\lambda_k}\weak \mu$ as Radon measures, where $\mu-\kappa_\infty e^{\lambda_\infty}$ is supported on $\{a_0,\dots,a_N\}$ and is therefore a linear combination of Dirac's deltas, supported on $\{a_0,\dots, a_N\}$.
\end{proof}

\begin{Lemma}[$\varepsilon$-regularity Lemma]\label{eLemma}
Let $\lambda\in L^1(S^1)$ be a solution of  
 \begin{equation}\label{liou6b}
 (-\Delta)^\frac{1}{2}\lambda=\kappa e^{\lambda}-1+c\delta_{-i},
\end{equation}
with $c\in \R$, $\bar\kappa:=\|\kappa\|_{L^\infty}<\infty$, $L:=\|e^{\lambda}\|_{L^1}<\infty$.
Assume that for some arc $A\subset S^1\setminus \{-i\}$
\begin{equation}\label{condgammaexpb}
 \int_A |\kappa| e^\lambda  d\theta  \le \pi-\ve\,, 
\end{equation}
for some  $\varepsilon>0$.
Then for every arc $A'\Subset A$ with
$\dist (A^c, A')=\delta$
\begin{equation}\label{estinfty}
\|\lambda-\bar \lambda\|_{L^{\infty}(A')} \le C(\delta,\varepsilon,\bar\kappa,L).
\end{equation} 
\end{Lemma}

\begin{proof} Set
$$f_1=(\kappa e^{\lambda}-1)\chi_A,\quad f_2=(\kappa e^\lambda-1)\chi_{A^c},$$
and define 
$$
\lambda_i(t):=G*f_i(t)=\int_{S^1}G(t-\theta)f_i(\theta)d\theta,\quad i=1,2,
$$
where $G$ is as in Lemma \ref{lemmafund4}. From  \eqref{eqF} and \eqref{rapr} it follows that 
\begin{equation}\label{lambdadec}
\begin{split}
\lambda-\bar \lambda&=G*(\kappa e^\lambda-1)+\frac{c}{2\pi}\log(2(1+\sin\theta))\\
&=\lambda_1 +\lambda_2+\frac{c}{2\pi}\log(2(1+\sin\theta)).
\end{split}
\end{equation}
 
Choose now an arc $A''$ with $A'\Subset A''\Subset A$ and $\dist(A'', A^c)=\dist(A',(A'')^c)=\frac{\delta}{2}$. With \eqref{Gdec} we easily bound
\begin{equation}\label{estu2}
\|\lambda_2\|_{L^\infty(A'')} \le C_1= C_1(\bar\kappa,L,\delta).
\end{equation}
It follows from \eqref{condgammaexpb} and Theorem \ref{MT4} that $\|e^{|\lambda_1|}\|_{L^p(S^1)}\le C_{p,\ve}$ for some $p>1$, and consequently also $e^{\bar \lambda}\le C$. Then for $t\in A'$ we have 
\[\begin{split}
\lambda_1(t)&\le\int_{A}G(t-\theta)(|\kappa| e^{\lambda_1(\theta)}e^{\lambda_2(\theta)+\bar \lambda}-1)d\theta\\
&\le \bar\kappa \bigg(e^{C_1+\bar \lambda}\underbrace{\int_{A''}G(t-\theta)e^{\lambda_1(\theta)}d\theta}_{(1)}+\underbrace{\int_{A\setminus A''}G(t-\theta)e^{\lambda(\theta)}d\theta}_{(2)}+C\bigg)\\
&\le C,
\end{split}\]
where in $(1)$ we use that   $G\in L^q(S^1)$ for $q\in [1,\infty)$ and in $(2)$ we use that $G\in L^{\infty}(A'\times (A\setminus A''))\,.$

Finally notice that for $e^{i\theta}\in A'$ the term $\frac{c}{2\pi}\log(2(1+\sin\theta))$ in \eqref{lambdadec} is bounded depending only on $\delta\ge \dist(A',\{-i\})$.
\end{proof}

\begin{Lemma}\label{L2infty} Let $\lambda: S^1\to S^1$ satisfy
$$(-\Delta)^\frac12 \lambda -c\delta_{-i}\in L^1(S^1)$$
for some $c\in\R$, and let $\tilde \lambda$ be the harmonic extension of $\lambda$ to $D^2$. Then
\begin{equation}\label{stimaLo1}
\|\nabla \tilde \lambda\|_{L^{(2,\infty)}(D^2)}\le C(\|(-\Delta)^\frac12 \lambda\|_{L^1(S^1)}+|c|),
\end{equation}
and for any ball $B_r(x_0)$
\begin{equation}\label{stimaLo2}
\frac{1}{r}\int_{B_r(x_0)\cap D^2}|\nabla \tilde \lambda|dx \le C\|\nabla \tilde \lambda\|_{L^{(2,\infty)}(B_r(x_0)\cap D^2)}.
\end{equation}
\end{Lemma}

\begin{proof}
Set $\lambda_1:=\lambda-\frac{c}{2\pi}\log2(1+\sin\theta)$, so that
$$(-\Delta)^\frac12 \lambda_1= (-\Delta)^\frac12 \lambda +c\delta_{-i}-\frac{c}{2\pi}=:f\in L^1(S^1).$$
We can then write
$$\tilde\lambda=\tilde \lambda_1 +\frac{c}{2\pi}\log(|z+i|^2),$$
where $\tilde\lambda_1$ and $\log(|z+i|^2)$ are the harmonic extentions to $\bar D^2$ of $\lambda_1$ and $\log(2(1+\sin\theta))$ respectively.

Then using \eqref{fraclapl7} from the appendix we write
 \begin{equation}\label{lambdal2infty}
\tilde\lambda_1(x)=\int_{S^1}G(x,y)\frac{\partial\tilde\lambda_1}{\partial \nu}(y)dy=\int_{S^1}G(x,y)f(y)dy\end{equation}
where $G$ is the Green function associated to the Neumann problem.
It is know that $\nabla_x G(x,y)\in L^{(2,\infty)}(S^1)$ (see e.g. \cite{CK}). Therefore $\nabla \tilde\lambda_1(x)\in L^{(2,\infty)}(D^2)\,$ and since we also have
$$|\nabla\log(|z+i|^2)|\le \frac{C}{|z+i|}\in L^{(2,\infty)}(D^2),$$
\rec{stimaLo1} follows\,.\par
  
The proof of \eqref{stimaLo2} follows from O'Neil's inequality \cite{Oneil}
$$\int_{A}|\nabla \tilde\lambda|dx \le \|\chi_A\|_{L^{(2,1)}(A)} \|\nabla \tilde \lambda\|_{L^{(2,\infty)}(A)}=\sqrt{|A|}\|\nabla \tilde \lambda\|_{L^{(2,\infty)}(A)}$$
for any $A\subset D^2$.
\end{proof}

\subsection{Blow-up and pinching analysis}

Let $\sigma_k:S^1\to S^1$ be a diffeomorphism fixing $-i$ such that $\gamma_k:= \Phi_k\circ \sigma_k$ is in normal parametrization, i.e. $|\dot\gamma_k|$ is constant for every $k$. Since the lengths $L_k$ of the curves $\gamma_k$ satisfy
$$\frac{2\pi}{\bar\kappa}\le L_k:=\int_{S^1}e^{\lambda_k} d\theta\le \bar L,$$ up to adding (upper and lower bounded) constants to $\lambda_k$ we can assume that $L_k=2\pi$ and
$$|\dot \gamma_k(z)|\equiv 1,\quad \text{for every }k\ge 1,\, z\in S^1\setminus\{-i\},$$
where 
$$\dot\gamma_k(z):=\frac{\de \gamma_k(e^{i\theta})}{\de \theta}\bigg|_{e^{i\theta}=z}.$$
This will not change the statement of Theorems \ref{mainth} and \ref{convergencephi}.
Also define an angle function $\varphi_k$ for $\gamma_k$ namely a continuous function $\varphi_k: S^1\setminus\{-i\}$ such that
\begin{equation}\label{defvarphik}
\dot \gamma_k=e^{i\varphi_k}.
\end{equation}
By the boundedness of the curvatures the following limits exist:
\begin{equation}\label{defanglepm}
\varphi_k(-i^\pm):=\lim_{t\to 0^\pm }\varphi_k(-i e^{it}).
\end{equation}
Up to a translation we can assume that $\gamma_k(-i)=0$, so that by Arzel\`a-Ascoli's theorem a subsequence converges to a $W^{2,\infty}$-curve with a possible angle at $-i$. Also for such curve an angle function $\varphi_\infty$ can be defined as above. Up to adding multiples of $2\pi$, or requiring that $\varphi_k(-i^+)\in [0,2\pi)$ for every $k$, we also have
\begin{equation}\label{convvarphik}
\varphi_k\to \varphi_\infty\quad \text{ uniformly in }S^1\setminus\{-i\}.
\end{equation}
Following \cite{DMR} we introduce the following distance function $D_k\colon S^1\times S^1 \to \R^+$.
\begin{eqnarray}\label{dk}
D_{k}(p,q)&=&\inf\bigg\{\left(\int_0^1 |\Phi'_k(\Delta(t))|^2|\Delta'(t)|^2dt\right)^\frac12,\nonumber\\
&&~~~ \Delta\in W^{1,2}([0,1],\bar D^2), ~\Delta(0)=\sigma_k(p), ~\Delta(1)=\sigma_k(q)\bigg\},
\end{eqnarray}
The infimum in (\ref{dk}) is attained by a path $\Delta$ such that
$$|\Phi'_k(\Delta(t))||\Delta'(t)|=const,$$
so that
$$\left(\int_0^1 |\Phi'_k(\Delta(t))|^2|\Delta'(t)|^2dt\right)^\frac12=\int_0^1 |\Phi'_k(\Delta(t))| |\Delta'(t)| dt=:\int_{\Delta_k} |\Phi'_k(z)||dz|.$$
In particular, when bounding $D_k(p,q)$ from above, it is sufficient to bound the \emph{length} of curves $\Phi_k\circ \Delta$. In the sequel we will often construct such curves by patching together different curves and arcs, and it is more convenient to estimate the length than the corresponding $W^{1,2}$-norms.


\begin{Definition}[Pinched point]\label{pinched}
A point $p\in S^1$ is called {\em pinched point} for the sequence $(\gamma_k)$ if there exists $q\in S^1$, $q\neq p$  such that 
$\lim_{k\to+\infty} D_k(p,q)=0$\,. We call $q$ ``conjugate" or ``pinched'' to $p$. We denote by  $\mathcal{P}$ the sets of the pinched points of $\gamma_{\infty}\,.$
\end{Definition}

\begin{Remark}
The definition of pinched point is independent of $\Phi_k$ and $\sigma_k$, as can be easily verified. Moreover being pinched is an equivalence relation.
\end{Remark}

\begin{Remark}
Contrary to the case studied in \cite{DMR} ($\beta_k=0$ for every $k$) in which every pinched point has a unique conjugate (\cite[Lemma 3.12]{DMR}), in our case a pinched point can have several conjugates (see e.g. Fig. \ref{Fpinch}).
\end{Remark}

\begin{Lemma}\label{pinchdist} Let $p^+\in S^1\setminus \{-i\}$ pinched to $p^-\in S^1\setminus \{-i\}$. Then the arc $\M A \subset S^1\setminus \{-i\}$ joining $p^+$ and $p^-$ has length $|\M A|\ge \frac{\pi}{\bar \kappa}$.
\end{Lemma}

\begin{proof} Consider the curve $\gamma_k|_{\M A}$. Then $\gamma_k$ has bounded curvature and length independent of $k$, equal to the length of the arc joining $p^+$ to $p^-$ (because we assumed that $L_k=2\pi$, so that $|\dot\gamma_k|=1$).  Moreover, since $p^+$ and $p^-$ are pinched, 
$$\lim_{k\to\infty}|\gamma_k(p^+)-\gamma_k(p^-)|=0.$$
We can now easily see that the length of $\gamma_k$ is at least $\frac{\pi}{\bar \kappa}+o(1)$ by considering the direction of the tangent vector $\dot\gamma_k:\M A \to S^1$. As $k\to\infty$, the arc covered by $\dot\gamma_k$ will have length at least $\pi-o(1)$ (compare for instance the proofs of Lemmata 3.10 and 3.11 in \cite{DMR}), which completes our proof.
\end{proof}

\begin{Remark} Contrary to the case $\beta_k=0$ studied in \cite{DMR}, two pinched points can be arbitrarily close to $-i$ (see Figure \ref{Fpinchclose}).
\end{Remark}

\begin{figure}
\begin{center}
\includegraphics[width=14cm]{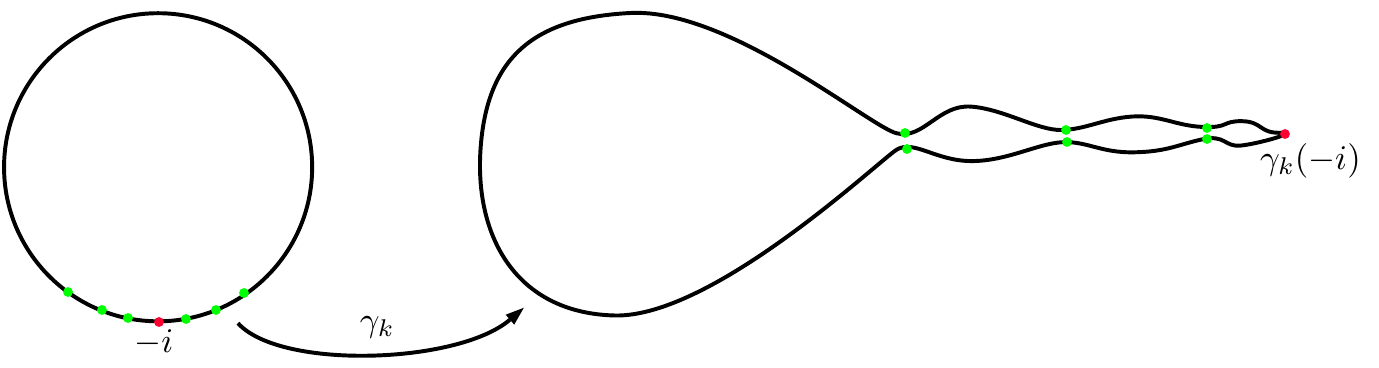}
 
\end{center}
\caption{\small Pinched points can be close to the angle at $-i$.}\label{Fpinchclose}
\end{figure}

\begin{Lemma} The set $\mathcal{P}$ of pinched points is closed in $S^1\setminus \{-i\}$.
\end{Lemma}

\begin{proof} Consider a sequence of points $p_n\to p\ne -i$ and for each of them let $q_n$ be a conjugate. Up to a subsequence we can assume that $q_n\to q\in S^1$ and by Lemma \ref{pinchdist} we have $q\ne p$. Then since
$$|\arc(p_n,p)|+|\arc(q_n,q)|\to 0\quad \text{as }n\to \infty,$$
it is easy, with a diagonal argument, to construct a sequence of paths $\Delta_k$ joining to $\sigma_k(p)$ to $\sigma_k(q)$ made of paths $\Delta_{n,k}$ joining $\sigma_k(p_n)$ to $\sigma_k(q_n)$ and small arcs joining $\sigma_k(p)$ to $\sigma_k(p_n)$ and $\sigma_k(q)$ to $\sigma_k(q_n)$, hence proving that $D_k(p,q)\to 0$, so that $p\in \M P$.
\end{proof}

\begin{Lemma}\label{lemmapinch} Assume that we are in case $2$ or Theorem \ref{convergencephi}. Then there are points $p_0,\dots, p_N\in S^1$ such that
$$\sigma_k^{-1}\to p_j\quad \text{in } L^\infty_{\loc}(\mathrm{arc}(a_j, a_{j+1})), \quad \text{for }j=0,\dots, N$$
(with the convention $a_{N+1}:=a_0=-i$). The points $p_j$ are pinched to each other and distinct, with the possible exception of $p_0$ and $p_N$ which might coincide, in which case $p_0=p_N=-i$ (compare to Fig. \ref{Fpinch}). If this is the case and $N=1$, in general it is not true that $-i$ is a pinched point.
Finally
$$\alpha_j= \varphi_\infty(p_{j})-\varphi_\infty(p_{j-1}).$$
In this formula, if $p_{j-1}=-i$, we replace it by $-i^+$ and if $p_j=-i$ we replace it by $-i^-$, where the definitions of $\varphi_{\infty}(-i^\pm)$ are given by \eqref{defanglepm}.
\end{Lemma}

\begin{figure}
\begin{center}
\includegraphics[width=12cm]{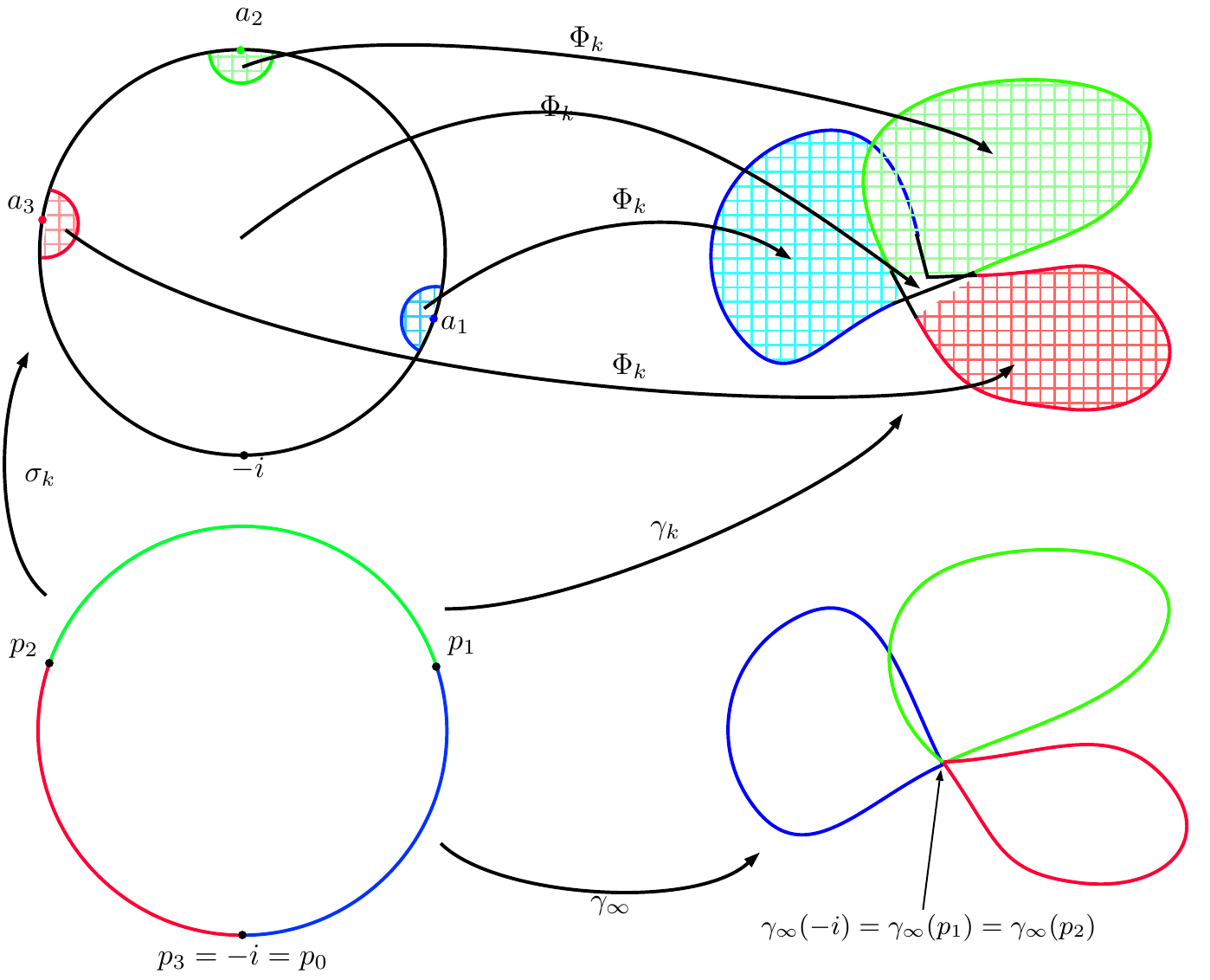}
\end{center}
\caption{\small An example of multiple pinching.}\label{Fpinch}
\end{figure}

\begin{proof} We have $\Phi_k\to Q$ in $W^{2,p}_{\loc}(\bar D^2\setminus \{a_0,\dots,a_N\})$, which implies (considering that $\gamma_k = \Phi_k\circ \sigma_k$ is arc-length parametrized) that
$$\frac{\de \sigma_k^{-1}(e^{i\theta})}{\de\theta}\to 0 \quad \text{uniformly locally in } S^1\setminus \{a_0,\dots, a_N\},$$
hence
$$\sigma_k^{-1}\to p_j\in S^1\quad \text{in }L^\infty_{\loc}(\arc(a_j,a_{j+1})),$$
for points $p_0,\dots, p_N$.
By \eqref{defB} and $\|\kappa_k\|_{L^\infty}\le \bar\kappa$ one has
\begin{align*}
\frac{\pi}{\bar\kappa}&\le \liminf_{k\to \infty}\int_{S_1\cap B_\delta(a_j)}e^{\lambda_k(z)} |dz|\\
&=\liminf_{k\to \infty}\int_{S_1\cap B_\delta(a_j)}|\Phi'(z)| |d z|\\
& =\liminf_{k\to\infty}\int_{\mathrm{arc}(p_{j-1},p_j)}|\dot\gamma_k|d\theta\\
&=|\mathrm{arc}(p_{j-1},p_j)|.
\end{align*}
Since $\sigma_k(-i)=-i$, and $a_0=-i$, we have that $p_\ell\ne p_j$ for $\ell\ne j$, except possibly the case $\ell=0$, $j=N$. To prove that the $p_j$ are pinched, fix $0\le j<\ell\le N$ with $p_j\ne p_\ell$ (unless $N=1$ and $p_0=p_1$, in which case we cannot prove that $p_0$ and $p_1$ are necessarily pinched), choose points $b_j\in \arc(a_j,a_{j+1})$, $b_\ell\in \arc(a_\ell,a_{\ell+1})$. Let $\Delta$ a path joining $b_j$ to $b_\ell$ (for instance a straigth line). Since $\sigma_k^{-1}(b_j)\to p_j$, we have
$$\int_{\arc(\sigma_k(p_j), b_j)} e^{\lambda_k(z)}|dz|=\int_{\arc(\sigma_k(p_j), b_j)} |\Phi_k'(z)||dz|\to 0,$$
and similarly for $\arc(\sigma_k(p_\ell), b_\ell)$. Joining these two arcs with $\Delta$ with obtain a path $\Delta_k$ with $\Delta_k(0)=\sigma_k(p_j)$, $\Delta_k(1)=\sigma_k(p_\ell)$ and length going to zero, so that $D_k(p_j,p_\ell)\to 0$ and $p_j$ is pinched to $p_\ell$.

For the last statement set $a_j^{\pm}= a_j e^{\pm i \delta}$ for $\delta>0$ sufficiently small, so that $a_j^+\in \arc(a_j,a_{j+1})$ and $a_{j}^-\in \arc(a_{j-1}, a_j)$. We have $\sigma_{k}^{-1}(a_j^+)\to p_j$ and $\sigma_{k}^{-1}(a_j^-)\to p_{j-1}$. Then, recalling \eqref{convvarphik}, we compute
\begin{align*}
\varphi_\infty(p_{j})-\varphi_\infty(p_{j-1})&=\lim_{k\to \infty}\varphi_k(\sigma_{k}^{-1}(a_j^+))-\varphi_k(\sigma_{k}^{-1}(a_j^-))\\
&=\lim_{k\to\infty}\int_{\arc(\sigma_{k}^{-1}(a_j^+),\sigma_{k}^{-1}(a_j^-))} \kappa_{\gamma_k}|d\gamma_k|\\
&=\lim_{k\to\infty}\int_{\arc(a_j^-,a_j^+)} \kappa_k e^{\lambda_k}d\theta\\
&=\alpha_j,
\end{align*}
where in the last identity we used \eqref{limk2}.
\end{proof}

We now introduce the following equivalence relation on the set $S^{1}\setminus(\mathcal{P}\cup\{-i\})$, which can geometrically understood by looking at Figure \ref{Fpinchequiv}.

\begin{Definition}\label{equiv}
 Given $q_1,q_2\in S^{1}\setminus (\mathcal{P}\cup\{-i\})$ we say that $q_1\sim q_2$  if
 there exists a sequence of paths $\Delta_k\colon[0,1]\to \bar D^2$ with $\Delta_k(0)=\sigma_k(q_1),
\Delta_k(1)= \sigma_k(q_2)$ such that
\begin{equation}\label{Dk>C}
\liminf_{k\to +\infty} d_k(\Delta_k,\sigma_k(\mathcal{P})) >0\,,\quad d_k(\Delta_k,\sigma_k(\mathcal{P})):=\inf_{\stackrel{t\in [0,1]}{p\in\M P}}d_k(\Delta_k(t),\sigma_k(p)),
\end{equation}
where $d_k:\bar D^2\times \bar D^2\to \R^+$ is the distance defined as
\[\begin{split}
d_{k}(z,w)=&\inf\bigg\{\left(\int_0^1 |\Phi'_k(\Delta(t))|^2|\dot \Delta(t)|^2dt\right)^\frac12, \\
&\qquad \Delta\in W^{1,2}([0,1],\bar D^2), ~\Delta(0)=z, ~\Delta(1)=w\bigg\}.
\end{split}\]
\end{Definition}

\begin{figure}
\begin{center}
\includegraphics[width=11cm]{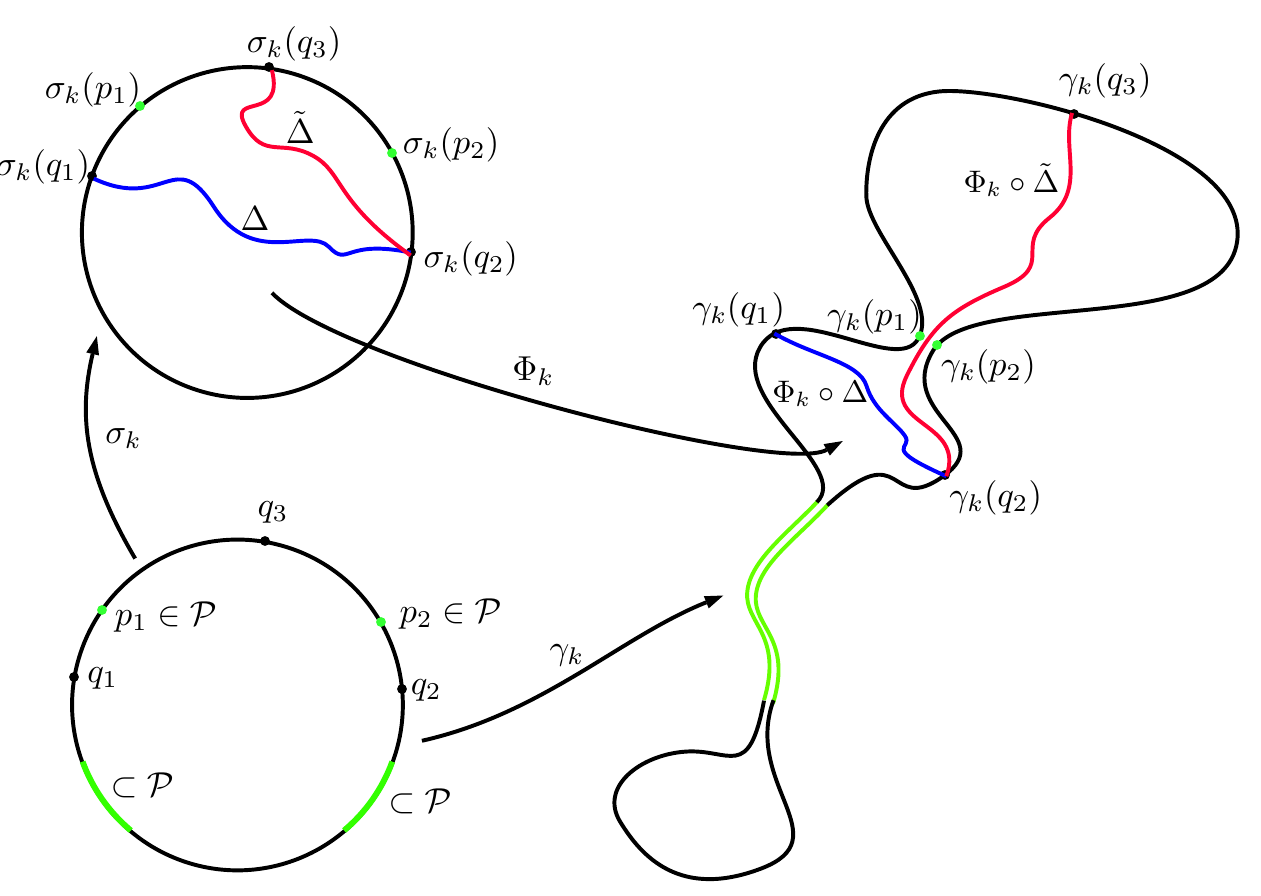}
 
\end{center}
\caption{\small We have $q_1\sim q_2$, but $q_2\not\sim q_3$ because any path $\tilde\Delta$ joining $\sigma_k(q_1)$ and $\sigma_k(q_3)$ passes close to the pinching at $\sigma_k(p_1)$ and $\sigma_k(p_2)$. Here $S^1\setminus\mathcal{P}$ has $4$ connected components and $3$ equivalence classes.}\label{Fpinchequiv}
\end{figure}

 \begin{Proposition}\label{eqcl}
 Let $q\in S^{1}\setminus (\mathcal{P}\cup\{-i\})$, $\mathcal{A}_q$ and $\mathcal{B}_q$ be respectively the    equivalence class and  the connected component   containing $q\,.$ Then $\mathcal{B}_q\subseteq\mathcal{A}_q$.
 \end{Proposition}

\begin{proof}
See Proposition 3.15 in \cite{DMR}.
\end{proof}

\begin{figure}
\begin{center}
\includegraphics[width=16cm]{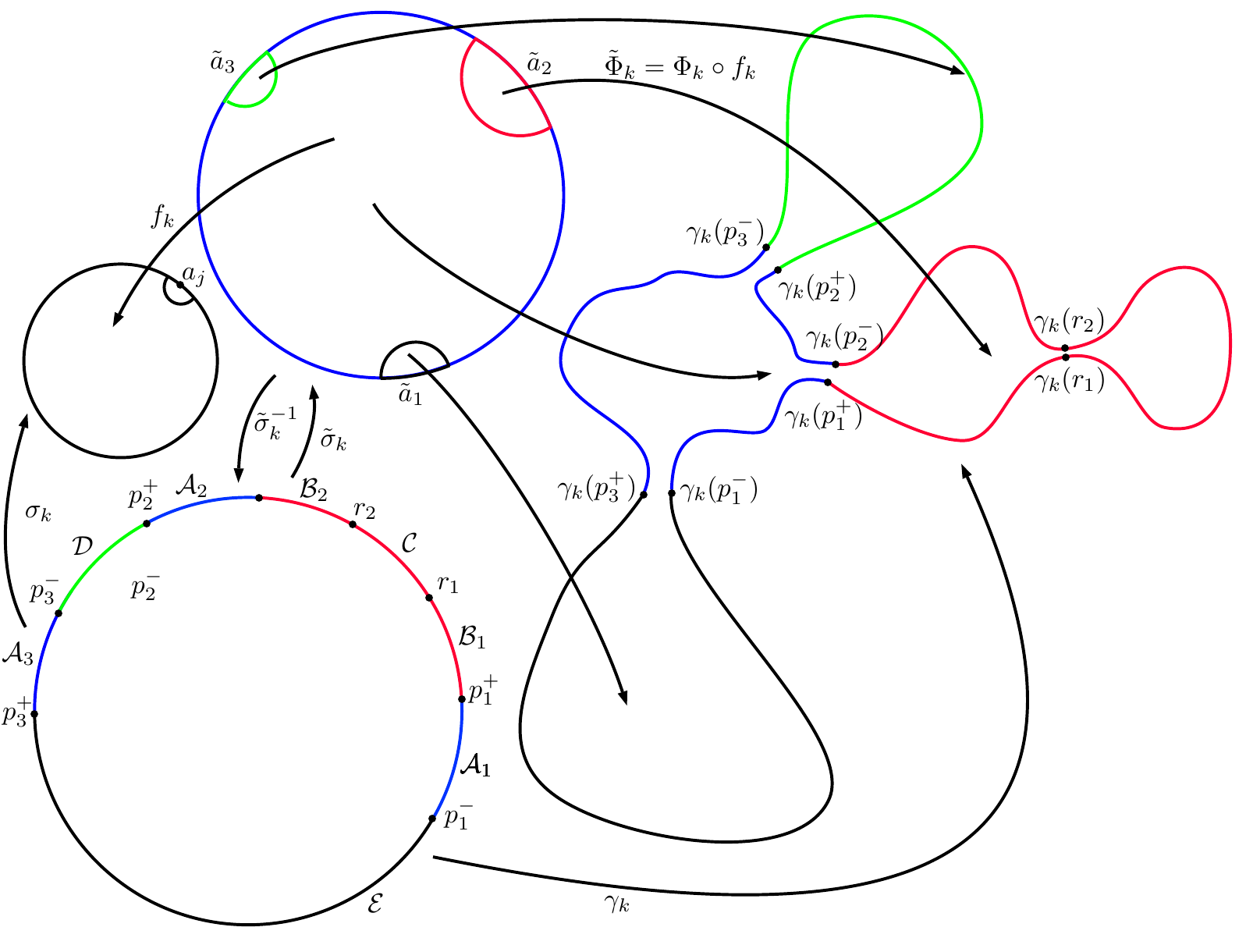}
\end{center}
\caption{\small An example of Proposition \ref{eqclbis} with $M=3$. In the limit $\tilde \Phi_k$ converges to an immersion $\tilde \Phi_\infty$ of $\bar D^2\setminus\{\tilde a_1,\tilde a_2,\tilde a_3\}$ into the region bounded by $\gamma_\infty(\M A)$. Here the equivalence class $\M A$ has the $3$ connected components $\M A_1,\M A_2,\M A_3$, while $\M B, \M C,\M D,\M E$ are the other equivalence classes of $S^1\setminus \M P$. In this example the pinched set is finite, but in general it can be countable, or even have positive measure (see Figure \ref{Fpinchequiv}). This figure also exemplifies Lemma \ref{lemmanest}: the connected component $\M B_1$ of the equivalence class $\M B$ is contained in the arc joining $\M A_1$ and $\M A_1$. Therefore the whole equivalence class $\M B$ is contained in the same arc between $\M A_1$ and $\M A_2$.}\label{Fmobius}
\end{figure}

\begin{Proposition}\label{eqclbis}
Let $\mathcal{A}$ be an equivalence class in $S^{1}\setminus (\mathcal{P}\cup\{-i\})$.
Then there exists a sequence $f_k\colon \bar D^2\to \bar D^2$  of M\"obius transformations such that
$\tilde \Phi_k:=\Phi_k\circ f_k\rightharpoonup \tilde\Phi_\infty$ in $W^{2,p}_{\loc}(\bar D^2\setminus \tilde B$), $\tilde B=\{\tilde a_1,\ldots,\tilde a_M\}$, and letting $\tilde \sigma_k$ be such that $\gamma_k=\tilde\Phi_k\circ\tilde\sigma_k$, one has $\tilde\sigma_k^{-1}\rightharpoonup\psi_\infty$ in $W^{2,p}_{\loc}(S^1\setminus \tilde B)$,
\begin{equation}\label{Acond}
\psi_\infty(S^1\setminus \tilde B)=\mathcal{A}
\end{equation}
and $\gamma_{\infty}(\mathcal{A})=\tilde\Phi_{\infty}(S^1\setminus \tilde B)$. In fact $(\gamma_\infty)_*[\mathcal{A}]=(\tilde \Phi_\infty)_*[S^1\setminus \tilde B]$.
In particular each equivalence class $\M A$ is a finite union of open arcs 
$$\M A_j=\psi_\infty(\arc(\tilde a_j,\tilde a_{j+1}))=\arc (p_j^-,p_j^+),\quad  j=1,\dots, M$$
(ordered in the counter-clockwise direction), with $p_j^+$ pinched to $p_{j+1}^-$ and $q_M^+$ pinched to $q_1^-$. See Figure \ref{Fmobius}.
\end{Proposition}

\begin{proof}
This follows with the same proof of Proposition 3.16 in \cite{DMR}.
\end{proof}

\begin{Lemma}\label{lemmanest}
The equivalence classes of $S^1\setminus \M P$ are nested, in the sense that if an $\M B_1$ is an arc of the equivalence class $\M B$ included in an arc $\M S$ joining two arcs $\M A_1$ and $\M A_2$ of the equivalence class $\M A$, then $\M B\subset \M S$, see e.g. Figure \ref{Fmobius}.
\end{Lemma}

\begin{proof}
This follows easily from the definition of equivalence classes, since assuming that $\M B$ contains an arc $\M B_2\not\in \M S$, any path $\Delta_k$ connecting $\sigma_k(\M B_1)$ to $\sigma_k(\M B_2)$ would cross any path $\tilde \Delta_k$ connecting $\sigma_k(\M A_1)$ to $\sigma_k(\M A_2)$. By assumptions one can choose such paths so that \eqref{Dk>C} is satisfied. Mixing these paths, one can easily find new paths joining $\sigma_k(\M B_1)$ to $\sigma_k(\M A_1)$ and still satisfying \eqref{Dk>C}, hence proving that $\M A=\M B$.
\end{proof}

\subsection{Proof of Theorem \ref{mainth}}

\begin{Lemma}\label{lemmaleqpi} In case 1 of Theorem \ref{convergencephi} we have $\alpha_j\le \pi$ for $1\le j\le N$.
\end{Lemma}

\begin{proof}
Assume that we are in case 1 of Theorem \ref{convergencephi}. Then up to a subsequence $\lambda_k\to\lambda_\infty$ in $W^{1,p}_{\loc}(S^1\setminus \{a_0,\dots,a_N\})$ and
\begin{equation}\label{defmu}
\kappa_k e^{\lambda_k}\weak \mu=\kappa_\infty e^{\lambda_\infty} +\sum_{j=0}^N\alpha_j\delta_{a_j},
\end{equation}
for some $\alpha_j\in \R$, $0\le j\le N$.
Choose a blow-up point $a_j\in\{a_1,\dots, a_N\}$. We want to prove that $\alpha_j\le \pi$.

\begin{figure}
\begin{center}
\includegraphics[width=15cm]{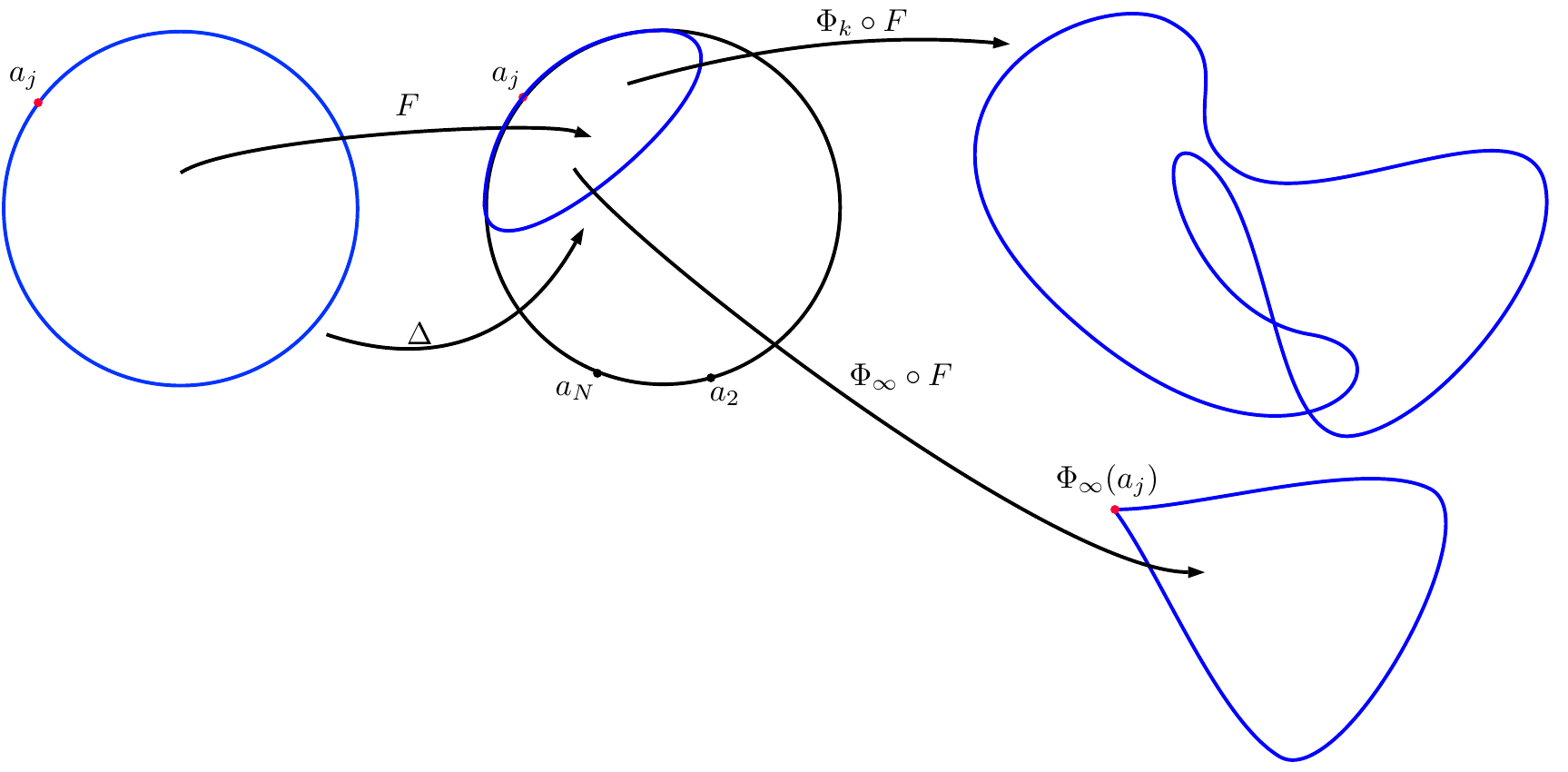}
 
\end{center}
\caption{\small Since $\phi_\infty$ can be extended to an immersion $F\circ \Phi_\infty$ only singular at $a_1$, one has $r(\phi_\infty)\ge 1$}\label{FDelta}
\end{figure}

Let $A=\arc(a_je^{-i\delta}, a_je^{i\delta})$, where $\delta$ is fixed but so small that $A\cap \{a_0,\dots,a_N\}=a_j$. Let now $\Delta: S^1\to \bar D^2$ be a given  regular $C^2$-Jordan curve with $\sigma(S^1)\cap \de \bar D^2= \bar A$, positively oriented. Assume also that $\Delta(z)=z$ for $z\in A$ (in other words $\Delta$ extends the arc $A$ to a regular Jordan curve). Consider now the closed curves $\phi_k:=\Phi_k\circ \Delta$. Since $\Delta$ can be extended to an immersion $F$ of $\bar D^2$ into $\bar D^2$, then $\phi_k$ can be extended to an immersion $\Phi_k\circ F$ of $\bar D^2$ into $\C$, with rotation index $r(\phi_k)=1$, see Definition \ref{defrot}. In particular, since $\phi_k$ does not have corners, it follows easily that
$$2\pi=\int_{S^1}\kappa_{\phi_k} |d\phi_k|,$$
where $\kappa_{\phi_k}$ denotes the curvature of $\phi_k$. On the other hand, our definitions and \eqref{defmu} imply that
$$\int_A \kappa_{\phi_k}|d\phi_k| =\int_A \kappa_k(z) e^{\lambda_k(z)}|dz|\to \int_A \kappa_\infty(z)e^{\lambda_\infty(z)}|dz|+\alpha_j\quad \text{as }k\to\infty.$$
Since away from $a_1$ the convergence of $\phi_k$ to $\phi_\infty:=\Phi_\infty\circ \Delta$ is regular enough, we obtain
$$2\pi=\int_{S^1}\kappa_{\phi_k} |d\phi_k| \to \int_{S^1\setminus \{a_j\}}\kappa_{\phi_\infty} |d\phi_\infty|+\alpha_j,$$
i.e.
$$\int_{S^1\setminus \{a_1\}}\kappa_{\phi_\infty} |d\phi_\infty|=2\pi-\alpha_j.$$
On the other hand, since $\phi_\infty$ can be extended to an immersion $\Phi_\infty\circ F\in C^1_{\loc}(\bar D^2\setminus \{a_1\}) \cap C^0(\bar D^2)$ with $J(\Phi_\infty\circ F)>0$ in $\bar D^2\setminus\{a_1\}$, by Theorem \ref{trmBlank} we infer that $r(\phi_\infty)\ge 1$. Looking at the definition of $r(\phi_\infty)$ (see Definition \ref{defrot} with $N=1$, $b_1=a_j$), we obtain that for a function $\varphi \in W^{1,1}(S^1\setminus \{a_j\})$ such that
$$\frac{\dot \phi_\infty}{|\dot \phi_\infty|}=e^{i\varphi},\quad \text{on }S^1\setminus \{a_j\}$$
(here $\dot \phi_\infty$ is the derivative with respect to the arc-length on $S^1$, taken with the usual counter-clockwise orientation), then
$$\varphi(a_j^-)- \varphi(a_j^+)=\int_{S^1\setminus \{a_j\}}\kappa_{\phi_\infty} |d\phi_\infty|= 2\pi-\alpha_j.$$
On the other hand
$$1\le r(\phi_\infty):=\frac{\varphi(a_j^-)- \varphi(a_j^+)+\ve}{2\pi} =\frac{2\pi-\alpha_j +\ve}{2\pi},$$
where $\ve\in [-\pi,\pi]$ is the exterior angle formed between $\dot \phi_\infty(a_j^-)$ and $\dot\phi_\infty(a_j^+)$. Then we conclude that
$$\alpha_j\le \ve\le \pi.$$
\end{proof}


\begin{Lemma}\label{lemmagepi} Assume that we are in case 2 of Proposition \ref{convergencephi}. Then $\alpha_j\ge \pi$ for $j=1,\dots,N$.  Moreover if for some $1\le j\le N$ and some $\ve>0$ we have
\begin{equation}\label{Kpos}
K_k|_{S^1\cap B_\ve(\alpha_j)}\ge 0,\quad \text{for every k},
\end{equation}
then $\alpha_j>\pi$.
\end{Lemma}

\begin{proof} According to Lemma \ref{lemmapinch} we have to prove that
$$\varphi_\infty(p_{j})-\varphi_\infty(p_{j-1})\ge \pi.$$
Let us first assume that for some $\delta>0$ we have
$$\arc(p_{j-1}, p_{j-1}e^{i\delta})\cup\arc(p_{j}e^{-i\delta}, p_{j})\subset S^1\setminus\mathcal{P}.$$
(That this is not always the case can be seen e.g. in Figure \ref{Fpinchcomp}, where two pinched points $p_1$ and $p_2$ belong to arcs of pinched points of positive measure.)  Since the equivalence classes in $S^1\setminus \mathcal{P}$ are nested by Lemma \ref{lemmanest}, we have that
$$\arc(p_{j-1}, p_{j-1}e^{i\delta})\cup\arc(p_{j}e^{-i\delta}, p_{j})\subset\M A$$
for an equivalence class $\M A$ of points in $S^1\setminus \M P$ (compare to Figure \ref{Fpinchcomp}, without the arcs $\M P_1$ and $\M P_2$, but it could also be that $\M A$ covers the whole $\arc(p_{j-1},p_j)$). Let $f_k$ and $\tilde\Phi_k=\Phi_k\circ f_k$ be as in Proposition \ref{eqclbis}, corresponding to this equivalence class $\M A$. Then $\tilde \Phi_k\to \tilde \Phi_\infty$ in $W^{2,p}_{\loc} (\bar D^2\setminus\{\tilde a_1,\dots,\tilde a_M\})$, and
$$\tilde \kappa_k e^{\tilde \lambda_k}\weak \sum_{j=1}^M \tilde \alpha_j \delta_{\tilde a_j},$$
where $\tilde \kappa_k=\kappa\circ f_k^{-1}$ and $\tilde \lambda_k= \log |\tilde\Phi_k'|_{S^1}|$.
Up to relabelling we can assume that $\M A_1=\psi_\infty (\arc(\tilde a_1,\tilde a_2))$ is the connected component component of $\M A$ starting at $p_{j-1}$ and $\M A_M=\psi_\infty (\arc(\tilde a_{M-1}, \tilde a_M))$ is the one ending at $p_{j}$ (see Figure \ref{Fmobius}). Then, by Lemma \ref{lemmapinch} applied to $\tilde \alpha_1$, we have (by the definition of $\varphi_\infty$)
$$\tilde\alpha_1=2\pi - (\varphi_\infty (p_{j})-\varphi_\infty(p_{j-1})),$$
which is to say that $\tilde \alpha_1$ is the amount of curvature, including the singularity $\beta_\infty\delta_{-i}$ at $-i$, accumulated along the curve from $\gamma_\infty(p_{j})$ to $\gamma_\infty(p_{j-1})$, which is $2\pi$ minus the curvature on accumulated along the curve from $\gamma_\infty(p_{j-1})$ to $\gamma_\infty(p_{j})$.
According to Lemma \ref{lemmaleqpi}, we have $\tilde \alpha_1\le \pi$,
hence
$$\alpha_j= \varphi_\infty(p_{j})- \varphi_\infty(p_{j-1})=2\pi -
\tilde \alpha_1\ge\pi. $$

\begin{figure}
\begin{center}
\includegraphics[width=16cm]{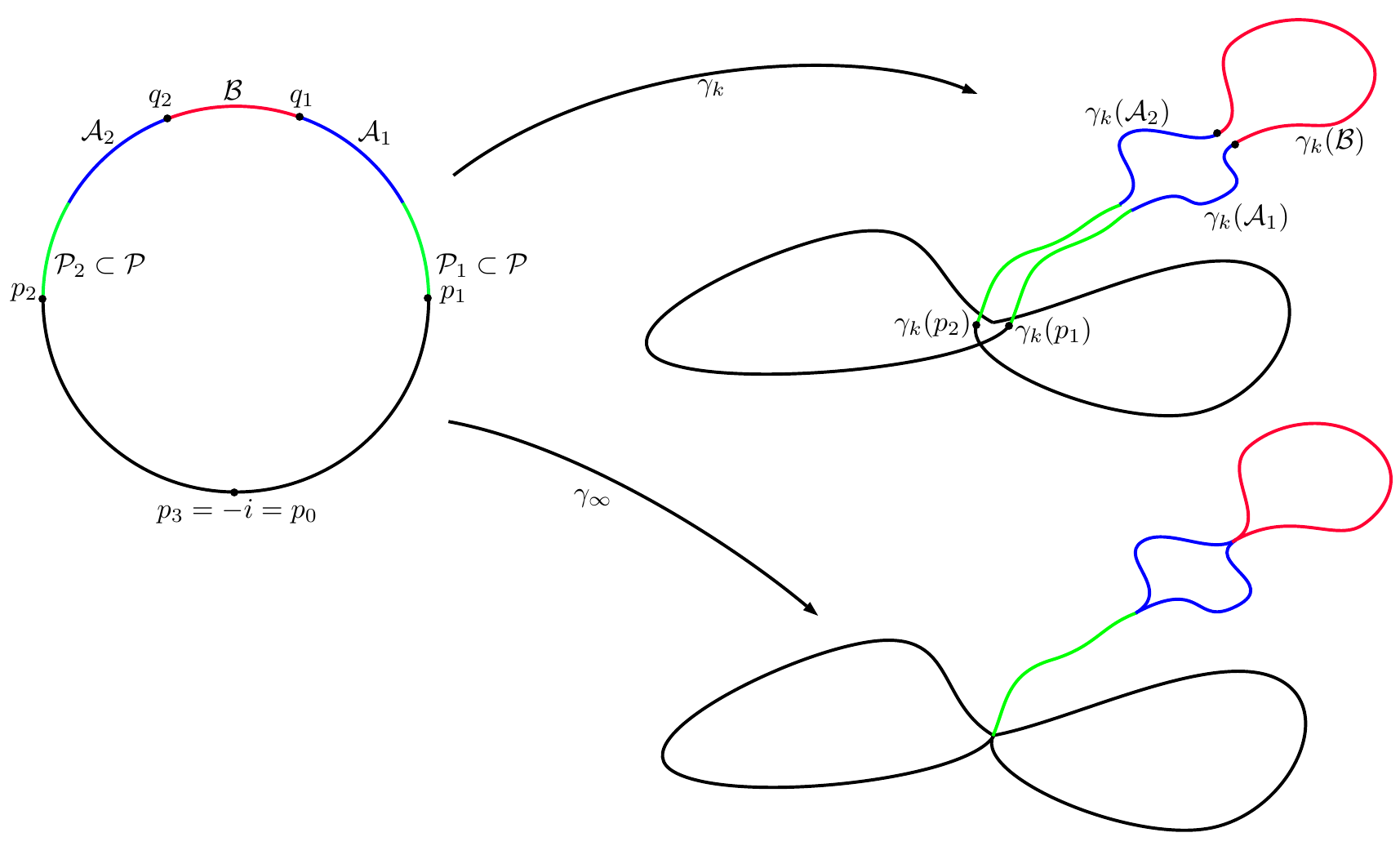}
\end{center}
\caption{\small The pinching set $\mathcal{P}$ can also contain arcs. Its complement contains the equivalent class $\mathcal{B}$ nested between $\mathcal{A}_1$ and $\mathcal{A}_2$.}\label{Fpinchcomp}
\end{figure}

This completes the proof, under the assumption that $p_{j-1}e^{it}$ and $p_{j}e^{-it}$ are not pinched for $t>0$ small. If they are pinched for \emph{every} $t\in (0,\delta)$ for a small $\delta>0$, one has $\gamma_\infty(p_{j-1}e^{it})=\gamma_{\infty}(p_{j}e^{-it})$ for $t\in (0,\delta)$ for a maximal $\delta>0$ (see e.g. Figure \ref{Fpinchcomp}). In this case
$$\int_{\arc(p_{j-1},p_{j-1} e^{i\delta})} \kappa_{\gamma_\infty}d\theta =- \int_{\arc(p_{j}e^{-i\delta},p_{j} )} \kappa_{\gamma_\infty}d\theta,$$
and these contributions to the value of $\alpha_j$ can be ignored.

Then we can reduce to the case in which a right neighbourhood of $p_{j-1}$ and a left neighbourhood of $p_{j}$ contain a sequence of connected components of $S^1\setminus \M P$. Since the equivalence classes of $S^1\setminus \M P$ are nested, we can find arcs $\M A_+$ arbitrarily close to $p_{j-1}$ and $\M A_-$ arbitrarily close to $p_{j}$, say $\dist (\M A_+, p_{j-1})\le \ve$, $\dist (\M A_-, p_{j})\le \ve$,  both belonging to an equivalence class $\mathcal A\subset S^1\setminus \M P$. Applying the previous reasoning (Proposition \ref{eqclbis} in particular) to $\M A$ one finds again $\tilde \alpha_1 \le \pi$, but because of the curvature lost in the arcs joining $p_{j-1}$ to $\M A_+$ and $p_{j}$ to $\M A_-$, which is at most $2\bar \kappa \ve$, one has
$$\alpha_j\ge 2\pi - \tilde \alpha_1 -2\bar\kappa \ve \ge \pi-2\bar\kappa\ve,$$
and since $\ve>0$ is arbitrarily small we conclude again that $\alpha_j\ge \pi$.

Finally, assume \eqref{Kpos}. Then, up to possibly taking $\ve>0$ smaller, so that
$$B_\ve(a_j)\cap \{a_1,\dots, a_N\}=\{a_j\},$$
we have
$$\Phi_k|_{B_\ve(a_j)\setminus B_{\ve/2}(a_j)}\to Q\in \C,$$
so that the curves $\Phi_k|_{S^1\cap B_\ve(a_j)}$ have the following properties:
\begin{enumerate}
\item They have non-negative curvatures uniformly bounded by $\bar\kappa$ and lengths uniformly bounded by $\bar L$;
\item They are almost closed, in the sense that, writing $S^1\cap B_\ve(a_j)=\arc (a_j^-,a_j^+)$, we have $\Phi_k(a_j^+)-\Phi_k(a_j^-)\to 0$.
\end{enumerate}
We now claim that these properties imply that for constant $C(\bar L,\bar \kappa)>0$ depending on $\bar L$ and $\bar\kappa$ one has
\begin{equation}\label{pi+}
\int_{\arc(a_j^-,a_j^+)} \kappa_k e^{\lambda_k}|dz|\ge \pi +C(\bar L,\bar \kappa)+o(1)
\end{equation}
with $o(1)\to 0$ as $k\to \infty$, which implies
$$\alpha_j\ge \pi +C(\bar L,\bar \kappa)>\pi.$$
In order to prove \eqref{pi+} one can argue by contradiction as follows. Assume that for a sequence of $W^{2,\infty}$-curves $s_k:[0,L_k]\to \C$, which we may assume arc-length parametrized with $L_k\le \bar L$ and non-negative curvatures $\kappa_{s_k}$ uniformly bounded by $\bar \kappa$, and with
$$|s_k(0)-s_k(L_k)|\to 0,$$ we have
$$\int_{0}^{L_k}\kappa_{s_k}(t)dt \le \pi +o(1),\quad \text{as }k\to\infty.$$
Up to translation we can assume that $s_k(0)=0$ and by Arzel\`a-Ascoli's theorem we have convergence of a subsequence to a $W^{2,\infty}$-curve $s_\infty:[0,L_\infty]\to \C$ with curvature $0\le \kappa_{s_\infty}\le \bar\kappa$ and
$$\int_{0}^{L_\infty}s_{\gamma_\infty}(s)ds \le \pi,\quad s_\infty(0)=s_\infty(L_\infty),$$
which is impossible (see e.g. Lemma 3.10 in \cite{DMR}).
\end{proof}

\medskip

\noindent\textbf{Proof of Theorem \ref{mainth} (completed).} We start by applying Theorem \ref{convergencephi}. If case 1 occurs, then by Lemma \ref{lemmaleqpi} we have $\alpha_j\le \pi$ for every $j=1,\dots,N$. We claim that in fact $\alpha_j=\pi$ for $j=1,\dots,N$. To do so, we compose $\Phi_k$ with M\"obius diffeomorphisms
$$f_k=\frac{z-t_k a_j}{1-t_k \bar a z},$$
with $t_k\to 1^-$ sufficiently slow, so that for $\tilde \Phi_k=\Phi_k\circ f_k$, $\tilde \kappa_k:=\kappa_k\circ f_k$, $\tilde \lambda_k=\log|\tilde \Phi'_k|_{S^1}|$ we have
\begin{align*}
\tilde\alpha_j&:=\lim_{r\to 0}\lim_{k\to\infty} \int_{S^1\cap B_r(a_j)} \tilde \kappa_k e^{\tilde \lambda_k} |dz|\\
&=\lim_{r\to 0}\lim_{k\to\infty} \int_{S^1\cap B_r(a_j)}  \kappa_k e^{ \lambda_k} |dz|\\
&=\alpha_j.
\end{align*}
Now $\tilde \Phi_k$ is as in case 2 of Theorem \ref{convergencephi} because $\Phi_k\to \Phi_\infty$ with $\Phi_\infty$ an immersion (away from $B$), while $f_k$ concentrates. It follows from Lemma \ref{lemmagepi} that $\tilde\alpha_j\ge \pi$, hence $\alpha_j=\pi.$ Moreover this cannot  happen if $K_k\ge 0$, since in this case Lemma \ref{lemmagepi} would imply $\alpha_j> \pi$. In particular, if $K_k\ge 0$ for every $k$ and case 1 of Theorem \ref{convergencephi} occurs, then $N=0$, i.e. there is no blow-up.

If case $2$ occurs, we apply directly Lemma \ref{lemmagepi} to obtain $\alpha_j\ge \pi$ for $1\le j\le N$, and in fact $\alpha_j>\pi$ if $K_k\ge 0$ in a neighbourhood of $a_j$.

The convergence of $v_k\to v_\infty$ and the formula for $v_\infty$ follow exactly as in the proof of Theorem \ref{convergencephi} of \cite{DMR}.
\hfill$\square$

\section{Immersions of the disk with boundary singularities}

\begin{Definition} A curve $\gamma\in C^1_{\loc}(S^1\setminus \{b_1,\dots,b_N\},\C)$ is called piecewise of class $C^1$ if $\gamma\in C^0(S^1,\C)$ and there exist
\begin{equation}\label{gammadot}
\dot \gamma(b_j^\pm):=\lim_{t\to 0^\pm} \frac{\de \gamma(b_je^{it})}{\de t}, \quad \text{for }1\le j\le N.
\end{equation}
\end{Definition}


\begin{Definition}[Rotation index]\label{defrot}
Given a closed piecewise $C^1$-curve 
$$\gamma \in C^1(S^1\setminus \{b_1,\dots, b_N\}, \C)$$
arc-length parametrized, with finitely many self-intersections, with $b_1,\dots,b_N$ ordered in the counter-clockwise direction, (and w.l.o.g. we can consider $N\ge 1$) define $\varphi_j$ as follows: 
\begin{enumerate}
\item Setting $b_{N+1}:=b_1$, we have $\varphi_j\in C^0(\arc(b_j,b_{j+1}))$ and
$$\dot \gamma= e^{i\varphi_j},\quad \text{in }\arc(b_j,b_{j+1}), \text{ for } 1\le j\le N.$$
According to \eqref{gammadot} the quantities
$$\varphi_j(b_j^+):=\lim_{t\to 0^+} \varphi_j(b_j e^{it}),\quad  \varphi_j(b_{j+1}^-):=\lim_{t\to 0^-} \varphi_j(b_{j+1} e^{it})$$
are also well defined.
\item For $1\le j\le N$ set $\ve_{j}:=\varphi_j(b_j^+)-\varphi_{j-1}(b_j^-)\in [-\pi,\pi]$ ($\varphi_0:=\varphi_N$); this condition uniquely determines $\ve_j$ (called exterior angle) and $\varphi_{j+1}$ given $\varphi_j$, unless $\dot \gamma(b_{j+1}^+)=- \dot \gamma(b_{j+1}^-)$, and if $\dot \gamma(b_{j+1}^+)=- \dot \gamma(b_{j+1}^-)$ we prescribe $\varphi_{j+1}(b_{j+1}^+)-\varphi_j(b_{j+1}^-)=\pi$ if the curve turns left at $\gamma(b_{j+1})$ and $\varphi_{j+1}(b_{j+1}^+)-\varphi_j(b_{j+1}^-)=-\pi$ if the curve turns right (see Figure \ref{Findex2}). Turning left or right is well defined since $\gamma$ has finitely many self-intersections, so it is injective in a sufficiently small neighbourhood of $b_{j+1}$.
\end{enumerate}
Then we define the rotation index of $\gamma$ as
\begin{align*}
r(\gamma)&:=\frac{\varphi_N(b_{N+1}^-)-\varphi_1(b_1^+)+\ve_{1}}{2\pi}\\
&=\frac{1}{2\pi}\sum_{j=1}^N (\varphi_j(b_{j+1}^-)-\varphi_j(b_j^+)+\ve_j)\in \mathbb{Z}.
\end{align*}

\begin{figure}
\begin{center}
\includegraphics[width=14cm]{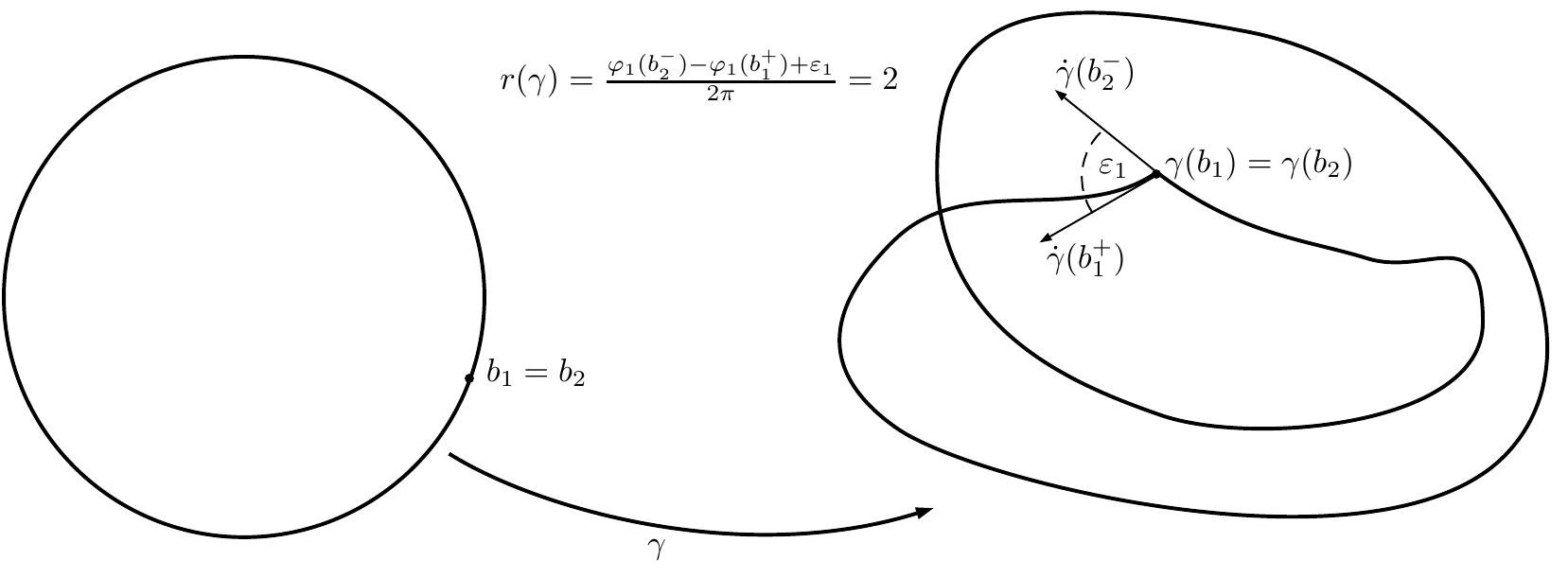}
 
\end{center}
\caption{\small Definition \ref{defrot}, case $N=1$. The curve has rotation index $2$.}\label{Findex2}
\end{figure}

A similar definition applies if $\gamma$ is regular ($|\dot \gamma| \ge \ve>0$ in $S^1\setminus \{b_1,\dots,b_N\}$) but not arc-length parametrized, by replacing $\dot \gamma$ with $\dot\gamma/|\dot\gamma|$, or if we replace $S^1$ by the image $\Gamma=f(S^1)$ of a piecewise $C^1$ Jordan curve $f\in C^1_{\loc}(S^1\setminus\{b_1,\dots, b_n\},\C)$, so that the rotation index of $\gamma: \Gamma\to \C$ is well defined as $r(\gamma\circ f)$.
\end{Definition}

The following theorem generalizes a result of S. J. Blank \cite{bla}. In the result of Blank an immersion $\Phi\in C^1(\bar D^2, \C)$ is considered and properties of the curve $\gamma:=\Phi|_{S^1}$ are studied, such as algebraic conditions on its self-intersections. If we allow that immersion $\Phi$ to be singular at finitely many points on the boundary, things become more comples. For instance it is not true anymore that $r(\Phi|_{S^1})=1$, but we can at least prove that $r(\Phi|_{S^1})\ge 1$, which is the statement of Theorem \ref{trmBlank} below. This will use many ideas of the original work of Blank, which we will recall during the proof.

\begin{Theorem}\label{trmBlank} Let $\Phi:\bar D^2\to \R^2$ satisfy the following:
\begin{enumerate}
\item $\Phi\in C^1_{\loc}(\bar D^2\setminus \{ b_1,\dots, b_N\})\cap C^0(\bar D^2)$ for a finite set of points $\{b_1,\dots, b_N\}\subset S^1$.
\item $J\Phi>0$ everywhere in $\bar D^2\setminus \{ b_1,\dots, b_N\}$.
\item $\gamma:=\Phi|_{S^1}\in C^1_{\loc}(S^1\setminus \{b_1,\dots, b_N\})$ is piecewise $C^1$ and regular.
\end{enumerate}
Then the rotation index of $\gamma$ is positive, i.e. $r(\gamma)\ge 1$. Moreover any word of Blank associated to $\gamma$ can be completely contracted (see definitions in the proof).
\end{Theorem}

\begin{proof}
Up to slightly deforming $\Phi$ we can assume that $\gamma$ has finitely many self-intersections, each of them is transversal and away from the points $\gamma(b_j)$, $1\le j\le N$. This of course does not change $r(\gamma)$. It is now possibly to locally deform $\Phi$ near each point $b_j$ in a way that $\gamma\in C^1(S^1)$, i.e. it is possibly to flatten the angles between $\dot \gamma(b_j^-)$ and $\dot\gamma(b_{j}^+)$, for instance by composing with suitable functions $\psi_j\in C^0(\C)\cap C^1(\C\setminus \{b_j\})$ with $\psi_j(z)=z$ away from $b_j$, in a way that the rotation number $r(\gamma)$ remains unchanged, therefore we can work with this stronger assumptions. 

Following the work of Blank \cite{bla} (see also \cite{poe} - \cite[Section 4]{MST}), to every bounded connected component of $\C\setminus \gamma(S^1)$ we associate a letter ($a,b,c,\dots$), and construct a segment ($\Sigma_a, \Sigma_b,\dots$) from the interior of the component and extending until it reaches the unbounded component of $\C\setminus \gamma(S^1)$. We can do that in such a way that the segment crosses $\gamma(S^1)$ in generic points (not at self-intersection points) and transversally. On each segment we move from the extreme in the interior of the corresponding component towards the other extreme and to each crossing with $\gamma(S^1)$ we associate a natural number (starting from $0$ for first crossing, then $1$, etc.) and an orientation ($+$ if $\gamma$ crosses the segment from the right, $-$ otherwise). For instance $\Sigma_a$ determines finitely many points on $\gamma(S^1)$ named $a_0^+, a_1^-,a_2^-,\dots$ (see Fig. \ref{FBlank}, where $\Sigma_a$ determines $a_0^-$ and $a_1^+$).

Since $\Phi\in C^0(\bar D^2)$ we can also assume that each segment stays away from $\Phi(B_\ve(b_j))$ for a small $\ve>0$ and $1\le j\le N$, so that $\Phi$ is a local $C^1$-diffeomorphism near the segments.

We construct now the ``word of Blank'' associated to $\gamma$ and this (non-unique) choice of segments by choosing a starting point for $\gamma$, following the curve and collecting all the letters corresponding to the crossing of the segments with $\gamma(S^1)$. For instance a word of Blank could be $a_0^- b_1^+ c_0^+ a_1^+b_0^+$ (as in Fig. \ref{FBlank}).

As in the work of Blank, we now contract such word as follows. Whenever a letter with minus sign appears (for instance $a_0^-$) in the above example, if we can find the same letter with plus sign so that between them (seeing the word in a periodic sense, i.e. up to a cyclic permutation: $a_0^- b_1^+ c_0^+ a_1^+b_0^+=b_1^+ c_0^+ a_1^+b_0^+a_0^-$) there is no other letter with minus sign, we cancel all such letters: for instance $b_1^+ c_0^+ a_1^+b_0^+a_0^-$ can be contracted to $b_1^+ c_0^+$ by removing $a_1^+$, $a_0^-$ and the letters in between. A word is completely contracted when no minus sign is left. In general a word of Blank associated to a curve cannot be completely contracted (to the curve $\infty$ one can associate the words $a_0^+b_0^-$ or $a_0^-b_0^+$ (see Fig. \ref{FBlank}), neither of whom can be contracted), but using that $\gamma=\Phi|_{S^1}$ for an immersion $\Phi$ as in the statement of the theorem we will contract the word of Blank  in a way that each contraction corresponds to a slice of the domain $\bar D^2$.

Indeed consider $\Phi^{-1}(\Sigma_a)$ for any segment $\Sigma_a$ as above. Because away from $\{b_1,\dots, b_N\}$ the map $\Phi$ is a locally invertible $C^1$-map, we see that $\Phi^{-1}(\Sigma_a)\cup\Phi^{-1}(\Sigma_b)\cup \dots $ is a union of $C^1$ curves $\sigma_{a,1},\sigma_{a,2},\dots, \sigma_{b,1},\dots$ with extremes, not intersecting, and with an orientation induced by the orientation of the segments $\Sigma_a,\Sigma_b,\dots$. Because of the orientation, all curves end on $S^1$ on points which are preimages of letters with positive sign, some curves start in the interior of $\bar D^2$, an those who start on $S^1$, start from preimages of points with negative sign. This implies that if negative signs exist (so the word of Blank is not completely contracted), there is at least a curve, say $\sigma_a$ starting and ending on $S^1$ splitting $\bar D^2$ into two compact components $U$ and $V$ (intersecting on $\sigma_a$), one of which, say $U$ does not contain preimages of points with negative sign (the proof of this fact is easy), and this leads to the first contraction of the word of Blank, by removing all letters in $\gamma(S_1\cap U)$.

Call $\Omega_1:=U$ and consider $\Phi|_{\Omega_1}$. A word of Blank associated to it, is easily seen to be given by the letters of the previous word of Blank with pre-image in $\gamma(S^1\setminus V)$, which by assumption have all positive sign. By deforming the domain $\Omega_1$ into a disk $\bar D^2$ (flattening the two angles at which $\sigma_a$ intersects $S^1$) we do not change the rotation index of $\Phi|_{\de \Omega_1}$, and by Lemma \ref{lemmapositive} below, we have that $r(\Phi|_{\de \Omega_1})\ge 1$.

We can now repeat the same procedure on $V$ if at least a letter with preimage in $S^1\setminus \Omega_1$ has negative sign. In finitely many step we end up with a decomposition $\bar D^2=\Omega_1\cup\dots\cup \Omega_n$ so that the 
$$r(\Phi|_{\Omega_j})\ge 1\quad  \text{for }1\le j\le n,$$
and we can also order them so that $\Omega_j\cap \Omega_{j+1}=\sigma_{\ell,j}\subset\Phi^{-1}(\Sigma_\ell)$ for some letter $\ell$. A brief look at the definition shows that
$$r(\Phi|_{\de(\Omega_j\cup \Omega_{j+1})}) = r(\Phi|_{\de \Omega_j})+r(\Phi|_{\de \Omega_{j+1}})-1$$
and similarly when gluing several components, so that finally
$$r(\Phi|_{S^1})= \sum_{j=1}^n r(\Phi|_{\de\Omega_j}) -(n-1)\ge 1.$$
This completes the proof.
\end{proof}

\begin{figure}
\begin{center}
\includegraphics[width=14cm]{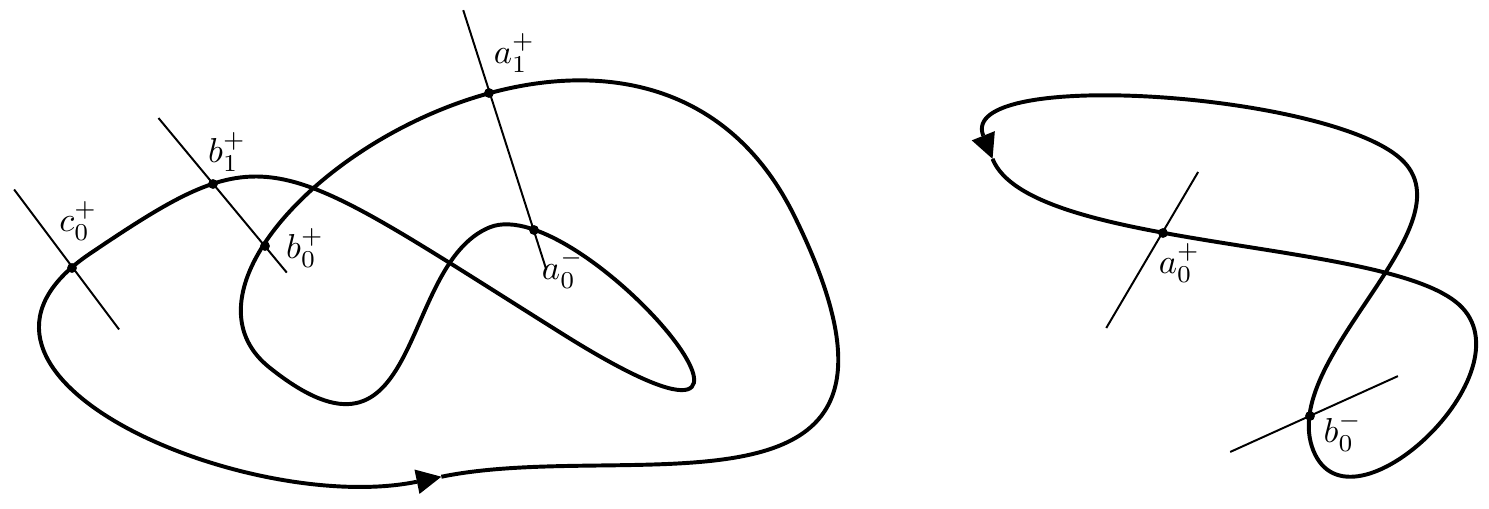}
 
\end{center}
\caption{\small Two words of Blank: $a_0^- b_1^+ c_0^+ a_1^+b_0^+$ and $a_0^+ b_0^-$. The first one contracts, so that the first curve can be extended to an immersion of the disk, the second one does not and the curve cannot be extended to an immersion of the disk.}\label{FBlank}
\end{figure}

\begin{figure}
\begin{center}
\includegraphics[width=14cm]{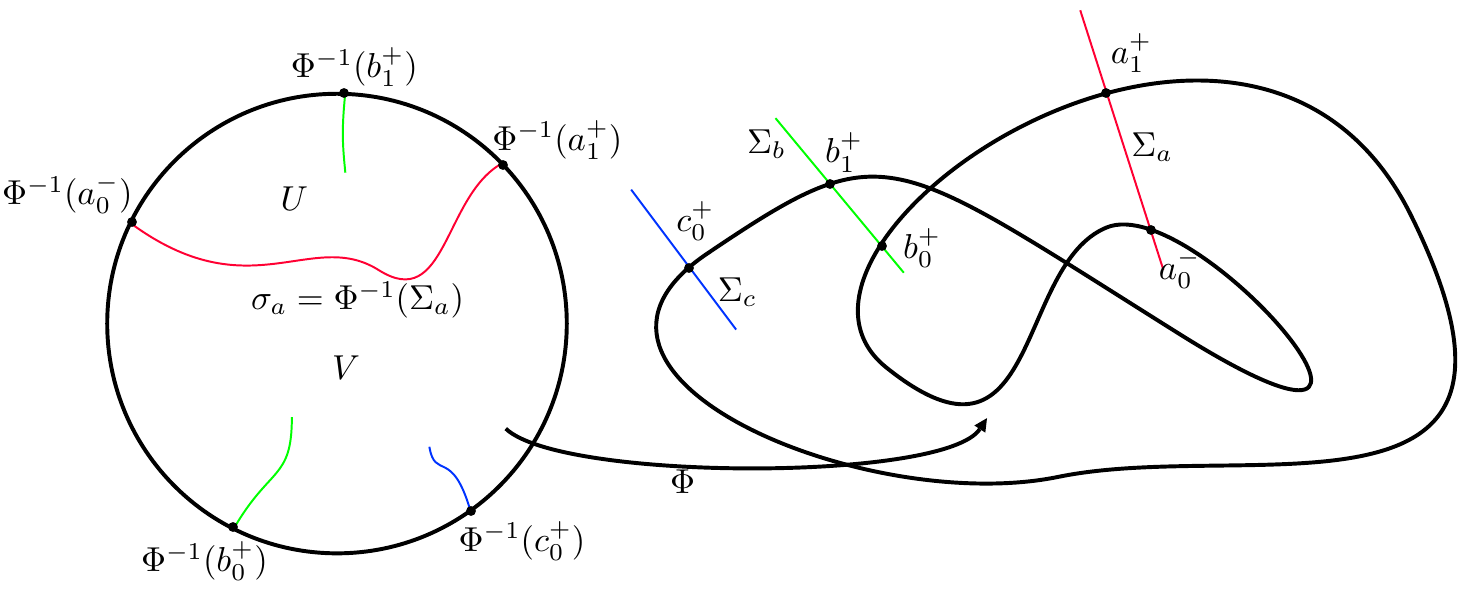}
 
\end{center}
\caption{\small The proof of Theorem \ref{trmBlank}.}\label{FBlank2}
\end{figure}

\begin{Corollary}\label{corangle} Let $\Phi$ and $\gamma$ be as in Theorem \ref{trmBlank}, with $N=1$ and let $\varphi=\varphi_1:S^1\setminus\{b_1\}\to \R$ be as in Definition \ref{defrot}. Then
$$\varphi(b_2^-)-\varphi(b_1^+)\ge \pi.$$
In particular if $\gamma|_{S^1\setminus \{b_1\}}\in W^{2,1}(S^1\setminus\{b_1\})$, then
$$\int_{S^1\setminus \{b_1\}} \kappa_\gamma d\theta\ge \pi.$$
As a consequence, given $\lambda\in L^1(S^1)$ solution of 
$$\lapfr \lambda = \kappa e^{\lambda}  -1 +\beta \delta_{a},$$
with $\kappa\in L^\infty(S^1)$,  $e^\lambda\in L^1(S^1)$, $a\in S^1$ and $\beta\in\R$, one always has $\beta\le \pi$.
\end{Corollary}

\begin{proof} Indeed by Theorem \ref{trmBlank} we have
$$1\le r(\gamma)=\frac{\varphi(b_2^-)-\varphi(b_1^+)+\ve_1}{2\pi},$$
with $\ve_1\in [-\pi,\pi].$ The last statement follows by first applying Theorem \ref{geomequiv} to construct $\Phi$ and $\gamma$, to which we can then apply the first part of the Corollary.
\end{proof}

\begin{figure}
\begin{center}
\includegraphics[width=14cm]{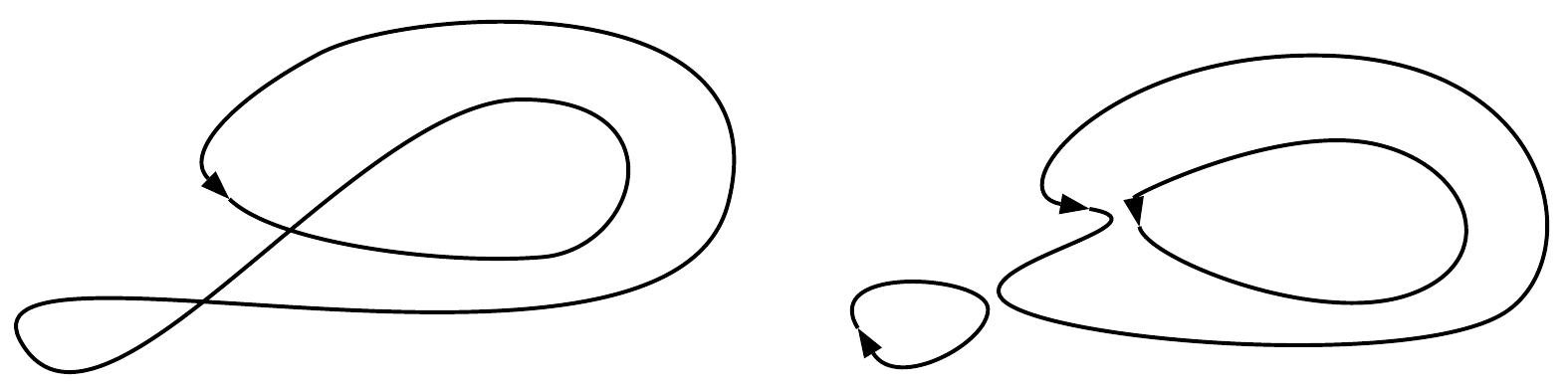}
 
\end{center}
\caption{\small The Seifert decomposition of an oriented curve with rotation index $1$ into $3$ simple oriented curves.}\label{Fseifert}
\end{figure}

\begin{Lemma}\label{lemmapositive} If a closed piecewise $C^1$-curve $\gamma\in C^1(S^1\setminus \{b_1,\dots,b_n\})$ admits a word of Blank without negative signs, then $r(\gamma)\ge 1$.
\end{Lemma}

\begin{proof} This (probably well-known) fact is based on what is called Seifert decomposition (Fig. \ref{Fseifert}). Up to a small perturbation of $\gamma$ we can assume that it has finitely many self-intersections, avoiding corners, and transversal. Decompose $\gamma(S^1)$ into finitely many piecewise $C^1$ Jordan curves $\gamma_j:S^1\to \C$, $j=1,\dots,m$ (with no self-intersections but new angles at the self-intersection points of $\gamma(S^1)$, which can be smoothed out, although this is not necessary in our setting) endowed with the same orientation given by $\gamma$. Each $\gamma_j$ is positively oriented, otherwise the segment leaving the component bounded by $\gamma_j(S^1)$ would give rise to a negative letter in the word of Blank. Then $r(\gamma_j)=1$ for every $j$ and one easily sees that
$$r(\gamma)=\sum_{j=1}^m r(\gamma_j)=m\ge 1.$$
\end{proof}

\appendix
 
We shall collect here some useful results and definitions about half-Laplacians and holomorphic maps.

\section{Appendix}\label{sec:lapl}
\subsection{The half-Laplacian on $S^1$ }\label{fraclaps1}

Given $u\in L^1(S^1)$ we define its Fourier coefficients as
$$\hat{u}(n)=\frac{1}{2\pi} \int_{S^1}u(\theta)e^{-in\theta}d\theta,\quad n\in\mathbb{Z}.$$
If $u$ is smooth we can define
\begin{equation}\label{fraclapl5}
\lapfr u(\theta)=\sum_{n\in\mathbb{Z}} |n|\hat{u}(n) e^{in\theta}.
\end{equation}
For $u\in L^1(S^1)$ we can define $\lapfr u\in \mathcal{D}'(S^1)$ as distribution as
\begin{equation}\label{fraclapl6}
\langle \lapfr u,\varphi\rangle := \int_{S^1} u\lapfr\varphi d\theta,\quad \varphi\in C^\infty(S^1).
\end{equation}

If $u\in L^1(S^1)$, set $\tilde u(r,\theta)$ (in polar coordinates) as
\begin{equation}\label{Poisson1}
\tilde u(r,\theta)=\frac{1}{2\pi}\int_0^{2\pi}P(r,\theta-t)u(t)dt,\quad P(r,\theta)=\sum_{n\in\mathbb{Z}}r^{|n|}e^{in\theta}=\frac{1-r^2}{1-2r\cos\theta +r^2}.
\end{equation}
This is the Poisson integral of $u$ and it is harmonic. Then 
\begin{equation}\label{fraclapl7}
\lapfr u=\frac{\de \tilde u}{\de r}\bigg|_{r=1}\quad \text{in }\mathcal{D}'(S^1), \quad \text{i.e.} \quad\left\langle \frac{\de \tilde u}{\de r}\bigg|_{r=1},\varphi\right\rangle := \int_{S^1} u\,\frac{\de \tilde \varphi}{\de r}\bigg|_{r=1} d\theta
\end{equation}
for every $\varphi\in C^{\infty}(S^1)$, where $\tilde\varphi$ is the harmonic extension of $\varphi\,$ in $D^2\,.$\par



For $u\in C^{1,\alpha}(S^1)$ there is also the following pointwise definition of $(-\Delta)^\frac12 u$:

\begin{Proposition} If $u\in C^{1,\alpha}(S^1)$ for some $\alpha\in (0,1]$, then $\lapfr u\in C^{0,\alpha}(S^1)$ and
\begin{equation}\label{fraclapl8}
\lapfr u(e^{i\theta})=\frac{1}{\pi}P.V.\int_0^{2\pi} \frac{u(e^{i\theta})-u(e^{it})}{2-2\cos(\theta-t)}dt.
\end{equation}
\end{Proposition}

\subsection{The half-Laplacian on $\R$}

For $u\in \mathcal{S}$ (the Schwarz space of rapidly decaying functions) we set
\begin{equation}\label{fraclapl0}
\widehat{(-\Delta)^\frac{1}{2}u}(\xi)=|\xi|\hat u(\xi),\quad \hat{f}(\xi):=\int_{\R}f(x)e^{-ix\xi}dx.
\end{equation}
One can prove that it holds
\begin{equation}\label{fraclapl}
(-\Delta)^\frac{1}{2} u(x)=\frac{1}{\pi} P.V.\int_{\R{}}\frac{u(x)-u(y)}{(x-y)^2}dy:=\frac{1}{\pi} \lim_{\varepsilon\to 0}\int_{\R{}\setminus [-\ve,\ve]}\frac{u(x)-u(y)}{(x-y)^2}dy,
\end{equation}
from which it follows that 
$$\sup_{x\in \R}|(1+x^2)(-\Delta)^{\frac{1}{2}}\varphi(x)|<\infty,\quad\text{for every }\varphi\in \mathcal{S}\,.$$
Then one can set
\begin{equation}\label{L12}
L_\frac{1}{2}(\R):=\left\{u\in L^1_{\loc}(\R):\int_{\R}\frac{|u(x)|}{1+x^2}dx<\infty   \right\},
\end{equation}
and for every $u\in L_{\frac{1}{2}}(\R)$ one defines the tempered distribution $(-\Delta)^\frac{1}{2}u$ as
\begin{equation}\label{fraclapl2}
\langle (-\Delta)^\frac{1}{2}u,\varphi\rangle :=\int_{\R} u(-\Delta)^\frac{1}{2} \varphi dx =\int_{\R}u\,\mathcal{F}^{-1}(|\xi|\hat \varphi(\xi))\,dx,\quad\text{for every }\varphi \in\mathcal{S}.
\end{equation}
An alternative definition of $\lapfr$ can be given via the Poisson integral. For $u\in L_{\frac{1}{2}}(\R)$ define the Poisson integral
\begin{equation}\label{Poisson2}
\tilde u(x,y):=\frac{1}{\pi}\int_{\R}\frac{yu(y)}{(y^2+(x-\xi)^2)}d\xi, \quad y>0,
\end{equation}
which is harmonic in $\R\times(0,\infty)$ and whose trace on $\R\times\{0\}$ is $u$.
Then we have
\begin{equation}\label{fraclapl3}
\lapfr u =- \frac{\de \tilde u}{\partial y}\bigg|_{y=0},
\end{equation}
where the identity is pointwise if $u$ is regular enough (for instance $C^{1,\alpha}_{\loc}(\R)$), and has to be read in the sense of distributions in general, with
\begin{equation}\label{fraclapl3b}
\bigg\langle -\frac{\de \tilde u}{\partial y}\bigg|_{y=0},\varphi\bigg\rangle:=\bigg\langle u, -\frac{\de \tilde \varphi}{\partial y}\bigg|_{y=0} \bigg\rangle,\quad \varphi\in\mathcal{S},\quad\tilde\varphi\text{ as in \eqref{Poisson2}}.
\end{equation}

More precisely:

\begin{Proposition}\label{lapeq} If $u\in L_{\frac{1}{2}}(\R)\cap C^{1,\alpha}_{\loc}((a,b))$ for some interval $(a,b)\subset\R$ and some $\alpha\in (0,1)$, then the tempered distribution $\lapfr u$ defined in \eqref{fraclapl2} coincides on the interval $(a,b)$ with the functions given by \eqref{fraclapl} and \eqref{fraclapl3}. For general $u\in L_\frac12(\R)$ the definitions \eqref{fraclapl2} and \eqref{fraclapl3} are equivalent, where the right-hand side of \eqref{fraclapl3} is defined by \eqref{fraclapl3b}.
\end{Proposition}

\begin{Lemma}[Fundamental solution of $(-\Delta)^\frac12$ on $S^1$]\label{lemmafund4} The function
$$G(\theta):=-\frac{1}{2\pi}\log(2(1-\cos (\theta)))$$
belongs to $ BMO(S^1)$,
can be decomposed as
\begin{equation}\label{Gdec}
G(\theta)=\frac{1}{\pi}\log\frac{\pi}{|\theta|} +H(\theta),\quad \theta\in [-\pi,\pi]\sim S^1, \quad \text{with }H\in C^0(S^1),
\end{equation}
and satisfies
\begin{equation}\label{eqF}
(-\Delta)^\frac12 G=\delta_1-\frac{1}{2\pi} \quad \text{in }S^1,\quad \int_{S^1}G(\theta)d\theta=0,
\end{equation}
and for every function $u\in L^1(S^1)$ with $(-\Delta)^\frac12 u\in L^1(S^1)$ one has
\begin{equation}\label{rapr}
u-\bar u = G*(-\Delta)^\frac{1}{2}u:=\int_{S^1} G(\cdot -\theta)(-\Delta)^\frac{1}{2}u(\theta)d\theta, \quad \text{for almost every }t\in S^1.
\end{equation}
\end{Lemma}

\begin{Theorem}\label{MT4} There exist constants $C_1,C_2>0$ such that for any $\ve\in (0,\pi)$ one has
\begin{equation}\label{stimaMT4}
C_1\le  \sup_{u=G*f\,:\, \|f\|_{L^1(S^1)}\le 1}\ve \int_{S^1}e^{(\pi-\ve)|u|}  d\theta \le C_2,
\end{equation}
and in particular
\begin{equation}\label{stimaMT4bis}
C_1\le  \sup_{\stackrel{u\in L^1 (S^1):\, \|(-\Delta)^{1/2} u-\alpha\|_{L^1(S^1)}\le 1 }{\text{for some }\alpha\in\R}}\ve \int_{S^1}e^{(\pi-\ve)|u-\bar u|}  d\theta \le C_2.
\end{equation}
\end{Theorem}

 \subsection{Useful results from complex analysis}

\begin{Lemma}\label{lemmahconst} Let $h\in C^0(\bar D^2,\mathbb{C})$ be holomorphic in $D^2$ with $h(S^1)\subset S^1$ and $0\not\in h( D^2)$. Then $h$ is constant.
\end{Lemma}

The following is a generalization of Lemma \ref{lemmahconst}.

\begin{Lemma}[Burckel \cite{bur}]\label{lemmabla} Let $h\in C^0(\bar D^2,\mathbb{C})$ be holomorphic in $D^2$ with $h(S^1)\subset S^1$ and $\deg h|_{S^1}=n\ge 0$. Then $h$ is a Blaschke product of degree $n$, i.e.
$$h(z)=e^{i\theta_0}\prod_{k=1}^{n} \frac{z-a_k}{1-\bar a_k z},\quad a_1,\dots, a_n \in D^2,\;\theta_0\in\R.$$
\end{Lemma}


     \end{document}